\documentclass[a4paper, 11pt, reqno, draft]{amsart}
\usepackage[utf8x]{inputenc}
\usepackage[title]{appendix}
\usepackage{amssymb,amscd, mathtools ,amsmath, commath}
\usepackage{amsfonts}
\usepackage{braket}
\usepackage{comment, fullpage, anysize}
\usepackage{enumerate}
\marginsize{1.2in}{1.2in}{1in}{1in}

\allowdisplaybreaks

\newtheorem{theorem}{Theorem}[section]
\newtheorem{proposition}{Proposition}[section]
\newtheorem{corollary}{Corollary} [section]
\newtheorem{mydef}{Definition}[section]
\theoremstyle{remark}
\newtheorem{remark}{Remark}[section]
\numberwithin{equation}{section}

\newcommand{\eps}{\varepsilon}
\newcommand{\diverg}{\textup{div}\,}
\newcommand{\dist}{\textup{dist}\,}
\newcommand{\plane}{\mathbb R^{2}}
\newcommand{\trace}{\textup{tr}}
\newcommand{\R}{\mathbb{R}}
\newcommand{\C}{\mathbb{C}}
\newcommand{\Z}{\mathbb{Z}}
\newcommand{\N}{\mathbb{N}}
\newcommand{\MExp}{\mathbb{E}}
\newcommand{\pP}{\mathbb{P}}
\newcommand{\sS}{\mathbb{S}}

\newcommand{\supp}{\text{supp }}

\newcommand{\F}{\boldsymbol{F}}
\newcommand{\gG}{\boldsymbol{G}}
\newcommand{\nnu}{{\bf{n}}}
\newcommand{\overbar}[1]{\mkern 1.5mu\overline{\mkern-1.5mu#1\mkern-1.5mu}\mkern 1.5mu}

\title{Vortices in a Stochastic Parabolic Ginzburg-Landau Equation}
\author{Olga Chugreeva}
\address{Lehrstuhl I f\"{u}r Mathematik RWTH Aachen University\\
Pontdriesch 14-16 52056 Aachen Germany}
\email{olga@math1.rwth-aachen.de}
\author{Christof Melcher}
\address{Lehrstuhl I f\"{u}r Mathematik  \& JARA Fundamentals of Future Information Technologies, RWTH Aachen University\\
Pontdriesch 14-16 52056 Aachen Germany}
\begin{document}


\maketitle

\begin{abstract}
We consider the variant of a stochastic parabolic Ginzburg-Landau equation that allows for the formation of point defects of the solution. The noise in the equation is multiplicative of the gradient type. We show that the family of the Jacobians associated to the solution is tight on a suitable space of measures. Our main result is the characterization of the limit points of this family. They are concentrated on finite sums of delta measures with integer weights. The point defects of the solution coincide with the points at which the delta measures are centered. 
\end{abstract}

\section{Introduction and statement of the results}
We study a stochastically perturbed complex Ginzburg-Landau equation in the regime when the solution develops point singularities. Our equation is a reduced version~\cite{BBH,SSVortices} of the full Ginzburg-Landau model for superconductivity~\cite{Landau}. The point singularities of the solution are a toy model for the Abrikosov vortices~\cite{Abrikosov}. By including a noise term into our equation we take into account unpredictable external fluctuations, which may affect both the formation and the behavior of the singularities.

For complex-valued functions $u_{\eps}$ that are defined on a smooth bounded simply connected domain $D\subset \plane$, we consider the equation 
\begin{equation}\label{SGL_Strat}
d u_{\eps}=\log(1/\eps)(\Delta u_{\eps}+(1/\eps^{2})(1-|u_{\eps}|^{2})u_{\eps})\,dt+
(\F\cdot\nabla) u_{\eps}\circ dB_{t}.
\end{equation}
It is a stochastically perturbed $L^{2}$-gradient flow of the Ginzburg-Landau energy functional 
\begin{equation}\label{GLEnergy}
E_\eps(u_{\eps}):=\int\limits_De_{\eps}(u_{\eps}) \,dx, \quad e_{\eps}(u_{\eps})=\frac{1}{2}|\nabla u_{\eps}|^2+\frac{1}{4\eps^2}(1-|u_{\eps}|^2)^2.
\end{equation}
The positive parameter $\eps$ is supposed to be small; we consider $\eps\in(0,1)$. The vector field~$\F=\F(x):D\to \R^{2}$ is suitably regular (at least $C^{3,\theta}$) and is compactly supported in the domain~$D$. We write $\F=(F^{1}, F^{2})$ and we identify~$u_{\eps}$ with an~$\R^{2}$-valued function~$(u_{\eps}^{1}, u_{\eps}^{2})$. The convective derivative associated to $\F$ is again a two-dimensional vector field, given by $(\F\cdot\nabla) u_{\eps}=(F^{k}\partial_{k} u_{\eps}^{1}, F^{k}\partial_{k} u_{\eps}^{2})$. The process $B_t$ is the standard one-dimensional Brownian motion. The equation is formulated in the Stratonovich sense.

We are particularly interested in the asymptotic behavior of $u_{\eps}$ for $\eps$ converging to zero. We focus on the energy regime 
\begin{equation}\label{Intro:EnergyWithVortices}
E_\eps(u_\eps)\leqslant C\log(1/\eps).
\end{equation}
Under~\eqref{Intro:EnergyWithVortices} and in the absence of noise, the asymptotic behavior of~$u_{\eps}$ is governed by the Ginzburg-Landau vortices of~$u_{\eps}$. Those are point singularities of~$u_{\eps}$ -- zeroes of~$u_{\eps}$ around which $u_{\eps}/|u_{\eps}|$ has a non-trivial winding number -- which persist in the limit $\eps\to 0$. In other words, vortices are finitely many distinct points in~$D$, each carrying a quantized topological charge. The energy concentrates at the vortex cores with a rate proportional to~$\pi\log(1/\eps)$. 

The vortex positions are best identified via the Jacobian of of $u_{\eps}$, 
\[J(u_{\eps}):=\det \nabla u_{\eps}.\]
If \eqref{Intro:EnergyWithVortices} holds, the family~$(J(u_{\eps}))$ is relatively compact in the space $(C_{0}^{0,\alpha}(D))^{*}$ with respect to the strong topology~\cite{JerrardSoner}. This space is the dual of the space of $\alpha$-H\"{o}lder continuous functions on $D$ that are zero on $\partial D$. Moreover, for any convergent subsequence we have that
\begin{equation}\label{Intro:ConvergenceOfJacobian}
J(u_{\eps})\xrightarrow[\eps\to 0]{} \pi\sum\limits_{k=1}^{N}d_{k}\delta_{a_{k}}.
\end{equation}
Here, $a_{k}$ is the position of the $k$th vortex and $d_{k}\in\Z$ is its topological charge (\textit{degree}). The number and degrees of the vortices are to some extent controlled by the energy: we have that 
$\pi\sum_{k=1}^{N}|d_{k}|\leqslant \sup_{\eps\in(0,1)}\log(1/\eps)^{-1}E_{\eps}(u_{\eps})$.

The estimate \eqref{Intro:EnergyWithVortices} immediately implies that the family of the \textit{rescaled energy densities }
\begin{equation}\label{eq:definitionOfRescaledEnergyDensity}
\mu_{\eps}(u_{\eps}):=(\log(1/\eps))^{-1}e_{\eps}(u_{\eps})
\end{equation}
is relatively compact in $(C_0^{0,\alpha}(D))^*$. 
However, without further assumptions we only have that 
$\mu_{\eps}(u_{\eps})$ converges to the measure~$\pi\sum\nolimits_{k=1}^{N}d_{k}^{2}\delta_{a_{k}}+\mu_{0},
$
with an unspecified $\mu_{0}$. Hence, the Jacobians indeed provide the most complete and natural description of the vortices. 

In this work, we focus on the Jacobians of the solution to~\eqref{SGL_Strat}. They are stochastic processes with values in certain spaces of measures. We show that the family of random measures $(J(u_{\eps}(t)))_{\eps}$ is tight for every positive~$t$. More importantly, we prove that its limit points are concentrated on the closure of the set 
 \[\Set{\pi\sum\nolimits_{k=1}^{N}d_{k}\delta_{a_{k}} | N\in \N, d_{k}\in \Z, a_{k}\in D}.
\]
This serves as the definition of the stochastic Ginzburg-Landau vortices. The vortices are the random points $a_{k}$ in the formula above.

We now state our results formally.
First, we prove existence and uniqueness of the solution to~\eqref{SGL_Strat}. We consider either Dirichlet or zero Neumann boundary conditions. 
\begin{theorem}\label{existence}
Suppose that for every $\eps$ in $(0,1)$, the initial data $u_{\eps}^{0}$ is deterministic and smooth. Then, for every~$\eps$ in~$(0,1)$ and every $T>0$, the initial-boundary value problem associated to \eqref{SGL_Strat} has a unique stochastically strong solution on $[0,T]$. The solution~$u_{\eps}(t, \cdot)$ is a continuous $C^{3, \beta}(\overbar D)$-valued process, for some $\beta\in(0,1)$.
\end{theorem}
We shall recall the definition of the strong solution in Section 2.3. In the same section, we specify the assumptions on the initial and boundary data and on the field $\F$. Theorem~\ref{existence} guarantees that $u_{\eps}$ is regular enough to have well-defined vortices. Our main results concern these vortices. 
We formulate them in the following two theorems.
\begin{theorem}\label{tightness}
Let $(u_{\eps}(t))$ be the family of solutions to~\eqref{SGL_Strat} provided by Theorem~\ref{existence}. Then, for every $t\in[0,T]$ and for every~$\alpha\in(0, 1]$, the families of the Jacobians~$(J(u_\eps(t)))_{\eps}$ and of the rescaled energy densities~$(\mu_\eps(u_\eps(t)))_{\eps}$ are tight on the space $(C_0^{0,\alpha}(D))^*$\!.  
\end{theorem}
Theorem~\ref{tightness} is the stochastic counterpart of the statement about the relative compactness of $(\mu_\eps(u_\eps))_{\eps}$ and $(J(u_\eps))_{\eps}$. Indeed, for the families of random variables, the relevant property is the weak relative compactness. It follows from the tightness, by virtue of the Prokhorov theorem. 

According to Theorem~\ref{tightness}, for every sequence $\eps_{n}\to 0$, the sequences $(J(u_{\eps_{n}}(t)))_{\eps_{n}}$ and~$(\mu_{\eps_{n}}(u_{\eps_{n}}(t)))_{\eps_{n}}$ weakly subconverge on the space $(C_0^{0,\alpha}(D))^*$. The limiting objects are two probability measures on  $(C_0^{0,\alpha}(D))^*$. We are able to characterize these measures. We now set the stage in order to present our last result.

We apply the Skorokhod representation theorem to the sequences $(J(u_{\eps_{n}}(t)))_{\eps_{n}}$ and $(\mu_{\eps_{n}}(u_{\eps_{n}}(t)))_{\eps_{n}}$. We can not do this on the space~$(C_0^{0,\alpha}(D))^*$, because it is not separable. However, it is continuously embedded into a separable space $W^{-1,q}$ for $q=2/(1+\alpha)$. The latter is the topological dual of the Sobolev space $W_{0}^{1,p}$ with~$1/p+1/q=1$. Now suppose that for $\eps_{n}\to 0$, the sequences $(J(u_{\eps_{n}}(t)))_{\eps_{n}}$ and $(\mu_{\eps_{n}}(u_{\eps_{n}}(t)))_{\eps_{n}}$ converge weakly on~$W^{-1,q}$, with limiting probability measures $\mathcal P_{J}$ and~$\mathcal P_{\mu}$. By the Skorokhod representation theorem, there exists a probability space $(\tilde \Omega, \mathfrak{\tilde F}, \tilde \pP)$ and on it, $W^{-1,q}$-valued sequences~$(\tilde J_{n})$, $(\tilde \mu_{n})$ with the following properties. The elements in the new sequences have the same laws as the corresponding elements in the original sequences. Namely, for every~$n\in \N$, we have $\mathcal L(\tilde J_{n})=\mathcal L( J(u_{\eps_{n}}(t)))$ and $\mathcal L(\tilde \mu_{n})=\mathcal L(\mu_{\eps_{n}}(u_{\eps_{n}}(t)))$, with $\mathcal L(\xi)$ denoting the law of a random variable $\xi$. The sequences converge $\tilde \pP$-almost surely in the topology of~$W^{-1,q}$ to the random variables~$\tilde J$ and~$\tilde \mu$, respectively. The laws of the limits are $\mathcal L(\tilde J)=\mathcal P_{J}$ and $\mathcal L(\tilde \mu)=\mathcal P_{\mu}$. We have detailed information on the structure of $\tilde J$ and~$\tilde \mu$.
\begin{theorem}\label{structure}
For $\tilde \pP$-almost all $\tilde\omega$, we have that
\[\tilde J(\tilde\omega)=\pi\sum_{k=1}^{N(\tilde \omega)}d_{k}(\tilde \omega)\delta_{a_{k}(\tilde \omega)}
\]
with $N(\tilde \omega)\in \N$, $d_{k}(\tilde \omega)\in \Z$ and $a_{k}(\tilde \omega)\in D$. 

For $\tilde \pP$-almost all $\tilde \omega$, $\tilde J(\tilde \omega)\ll \tilde \mu(\tilde \omega)$ as measures on $D$ and $\big|\frac{d\tilde J(\tilde \omega)}{d\tilde \mu(\tilde \omega)}(x)\big|\leqslant 1$, for $\tilde \mu(\tilde \omega)$-almost all $x$ in $D$. 

Finally,
\begin{equation}\label{eq:ExpectationOfNumberOfVortices}
\tilde\MExp\big[\pi \sum\nolimits_{k=1}^{N(\tilde \omega)}|d_{k}(\tilde \omega)|\big]\leqslant \sup_{n}\MExp\big[|\log\eps_{n}|^{-1}E_{\eps_{n}}(u_{\eps_{n}}(t))\big]\leqslant C,
\end{equation}
for a constant $C$ that is independent of $\eps$.
\end{theorem}
Theorem~\ref{structure} shows that the definition of the vortices in terms of the Jacobian remains consistent in the stochastic case. In the limit, the Jacobians do produce a random set of points with integer weights assigned to them. This result is the very first step towards the study of the vortex dynamics in the stochastic setting.

In the deterministic setting, the Ginzburg-Landau vortices have been extensively studied since early 1990es. Their behavior is quite well understood by now. To a collection of vortices, we associate the renormalized energy~$W(a_{1}, ..., a_{N})$. Essentially, $W$ is the Kirchhoff-Onsager energy of a vortex configuration in an incompressible fluid, cf.~\cite{MarchioroPulvirenti}, Chapter~4. If $u_{\eps}$ solves an equation based on $E_{\eps}$, its vortices solve an equation of the same type based on~$W$. For instance, the vortex configuration of 
a minimizer of~$E_{\eps}(u_{\eps})$ minimizes $W(a_{1}, ..., a_{N})$~\cite{BBH}. For the accelerated gradient flow of $E_{\eps}$, the vortex dynamics is governed by the gradient flow of $W$~\cite{E,LinHeat,JS_Dynamics,SSProd}. Similar facts have been proven for the Hamiltonian~\cite{Neu,CollianderJerrard} and the mixed gradient-Hamiltonian~\cite{KMMS,Miot} flows of the Ginzburg-Landau energy. Note that all these results are obtained for vortices of degrees~$\pm 1$. One needs this additional assumption because vortices of higher degree are believed to be unstable. Since the energy of a vortex scales quadratically in degree, it is energetically more favorable to have  many vortices of degree $\pm1$ than one of a higher degree. Still, a splitting of a vortex can not be adequately described by the methods existing so far. In this work, we do not derive the vortex motion law for~\eqref{SGL_Strat} and therefore we do not make any restriction on the vortex degrees.

The Jacobian is the crucial ingredient of the arguments that back the results on the vortex dynamics. Moreover, the use of the Jacobian has been considerably extended in the time-dependent case. The suitable object here is the \textit{total Jacobian}~\cite{SSProd,SSGamma}. It comprises both the usual Jacobian and the analogous quantities that depend on the time-derivative of $u_{\eps}$. The total Jacobian is compact if both the Ginzburg-Landau energy and the kinetic energy of $u_{\eps}$ are of the order $\log(1/\eps)$. The total Jacobian concentrates on the \textit{vortex paths} in the space $C([0, T], D)$. Therefore, it provides information on the vortex dynamics as a whole. This compactness result also gives some information on the regularity of the vortex trajectories. The actual equation on these trajectories is obtained as follows. From the equation on $u_{\eps}$, we derive the evolution equation on either $J(u_{\eps})$ or $\mu_{\eps}(u_{\eps})$. When we send $\eps$ to zero in these intermediate equations, we get the vortex motion law.

It is not clear, what could be the counterpart of the total Jacobian in the stochastic setting. In the construction of the total Jacobian, the time is treated essentially as a one more spatial variable. Furthermore, the time-derivative~$\partial _{t}u_{\eps}$ plays an independent and important role in this regard. The kinetic energy is the $L^{2}$-norm of $\partial _{t}u_{\eps}$. The stochastic situation is quite different. First, the time variable is utterly distinct from the spatial variables. Second, the time-derivative of the solution does not have a meaning on its own. 

In~\cite{Tch}, we derive the vortex motion law for a \textit{non-randomly perturbed} mixed Ginzburg-Landau flow. We fully understand the effect of the convective forcing in this case. We can adapt our conclusions to the slightly simpler case of the forced gradient flow of~$E_{\eps}$. We perturb the gradient flow by the convective term $(\gG\cdot\nabla)u_{\eps}$ with a smooth vector field~$\gG=\gG(x,t)$. The resulting equation on $u_{\eps}$ is the deterministic counterpart of \eqref{SGL_Strat}. For it, the motion of the vortices is governed by the forced gradient flow of $W$. The forcing is equal to~$\gG(a_{k}(t), t)$. The external field $\gG$ is thus transferred directly into the equation on the vortices. 

It would be very interesting to derive the vortex motion law for \eqref{SGL_Strat}. By analogy with the deterministic case, we expect that the law is given by a system of stochastic ODEs 
\begin{equation}\label{eq:StochVortexMotion}
d{a}_{k}=-\frac{1}{\pi}\partial_{a_{k}}W(a_{1}, ... , a_{N})\,dt+\F(a_{k}(t))\circ dB_{t}.
\end{equation}
The system is the \textit{randomly forced} gradient flow of $W$. The complete verification of this conjecture exceeds the scope of the present work. We can, nevertheless, support our hypothesis by one elementary example. We consider~\eqref{SGL_Strat} on a flat torus and with a \textit{constant} field $\F$. In this case, the initial data must have vortices of degrees $\pm 1$. The random change of variables $x\mapsto x-\F B_t$ converts the equation~\eqref{SGL_Strat} into the deterministic parabolic Ginzburg-Landau equation. To this equation, we apply the known theory. We pass to the vortex equation and then transform the variables back. This approach indeed yields a stochastic vortex motion law of the desired form \eqref{eq:StochVortexMotion}. 

The general framework of our investigation can be described as follows. We start with a family of parabolic PDEs indexed by a small parameter $\eps$. The asymptotic behavior of the solution for $\eps$ going to~$0$ is governed by the set of topological defects of the solution. On this set the solution takes, for topological reasons, values that are strongly penalized by the associated energy functional. The set of topological defects persists as $\eps$ goes to zero. Moreover, it coincides with the concentration set of a certain quantity related to the solution. We study the topological defects of the solution to the corresponding stochastically perturbed PDEs with two goals. First, we want to show that the set has the same geometric structure as in the unperturbed case. Second, we want to show that the dynamics of the set is governed by a certain randomly forced equation. The equation is guessed from the deterministic case. 

For the stochastic Ginzburg-Landau equation, we have achieved the first goal; this is the content of Theorems~\ref{tightness} and~\ref{structure}. Regarding the second goal, we have at least a conjecture. R\"{o}ger and Weber~\cite{RW} have addressed the same problem in the context of the Allen-Cahn equation. Their results are very similar to ours. The present work is strongly inspired by~\cite{RW}. 

The Allen-Cahn equation is the one-dimensional counterpart of the parabolic Ginz\-burg-Landau equation (cf. a detailed review~\cite{BethuelOrlandiSmets}). For this equation, the unknown function $U_{\eps}$ is scalar, but the dimension of its domain can be arbitrary. The underlying energy is, up to the scaling in $\eps$, the same as \eqref{GLEnergy}. The typical energy behavior implies relative compactness of the family of the solutions itself. 
For $U_{\eps}$, the set of topological defects is a hypersurface that is identified with the concentration set of $|\nabla U_{\eps}|$. Its evolution is governed by the mean curvature flow (MCF). 

In~\cite{RW}, the Allen-Cahn equation is perturbed by the noise of the same form as in~\eqref{SGL_Strat}. The authors show that the family of the solutions is tight on the space $C([0,T], L^1(D))$, natural for this problem. Moreover, they characterize the limit of any weakly convergent sequence of the solutions. It is, for almost every $t\in[0,T]$, concentrated on the set of phase-indicator functions of bounded variation. It can thus be associated at least to random Caccioppoli sets, if not to smooth hypersurfaces. The authors provide arguments in favor of the intuition that the sets perform a stochastically perturbed MCF. 

In most aspects, our proofs adapt the reasoning of~\cite{RW} to the Ginzburg-Landau case. In particular, our conclusions are valid for the forcing of a much more general form, namely,~ 
$
\sum_{k=1}^{\infty} \F_k(x)\circ dB_t^k.
$
Here, $\F_k$ are smooth vector fields and $B_t^k$ are independent one-dimensional Brownian motions. We consider the simplest possible case of a single Brownian motion to keep our presentation transparent. 

The proof of Theorem~\ref{existence} is essentially the same as for the Allen-Cahn equation. Our proof relies on Kunita's theory of stochastic 
flows~\cite{Kunita}. We perform a change of variables that transforms \eqref{SGL_Strat} into a family of parabolic PDEs. In these PDEs, the coefficients are random, but analytically well-behaved. Therefore, the existence, uniqueness and regularity of the solution follow from the standard theory. This technique generalizes the change of variables that we use in the case of constant $\F$. The idea of studying a stochastic differential equation by transforming it to a random family of deterministic equations is quite general. In the case of stochastic ordinary differential equations, it goes back to Doss~\cite{Doss} and Sussmann~\cite{Sussmann}. In the case of stochastic PDEs, it has been used in a wide variety of settings as well (see~\cite{daPrato_Tubaro_Semilinear,Brzezniak_Capinski_Flandoli,Flandoli_Lisei,Brzezniak_ItoUMD,Filipovic_Tappe_Teichmann,Barbu_Brzezniak_Hausenblas_Tubaro}, the list is far from being complete).

Theorems \ref{tightness} and~\ref{structure} are stochastic counterparts of the analytical results~\cite{JerrardSoner} on the Jacobians. We apply these results at almost every element of the underlying probability space. The main task is to show that this application is eligible. 

In the deterministic case, the crucial step for the proof of compactness is the decomposition (\cite{JerrardSoner}, Proposition~3.2) 
\begin{equation*}
J(u_{\eps})=J_0^\eps+J_1^\eps.
\end{equation*}
The term $J_0^\eps$ is bounded uniformly in~$\eps$ with respect to the topology of~$(C_{0}^{0})^{*}$. Its $(C_{0}^{0})^{*}$-norm is controlled by~$(\log(1/\eps))^{-1}E_{\eps}(u_{\eps})$. For the term $J_1^\eps$, we obtain uniform bounds only in the distributional norm. However, its $(C_{0}^{0,1}(D))^{*}$-norm vanishes as $\eps$ goes to zero. This reflects the situation described by~\eqref{Intro:ConvergenceOfJacobian}. In~\eqref{Intro:ConvergenceOfJacobian}, the limiting object is a finite measure, but the convergence takes place in a strictly weaker norm.

The structure of the limit points is due to the relation between $J(u_{\eps})$ and $\mu_{\eps}(u_{\eps})$. The limit points of the Jacobian are supported on the set where $\mu_{\eps}(u_{\eps})$ concentrates. According to~\eqref{Intro:EnergyWithVortices}, the rescaled energy measure can concentrate only at a finite number of points. Here, it is crucial that the families $(J(u_{\eps}))_{\eps}$ and~$(\mu_{\eps}(u_{\eps}))_{\eps}$ stem from the same family of functions~$(u_{\eps})$. 

The claim of Theorem~\ref{tightness} is, as in the deterministic case, tied to the energy behavior. The key step in the proof is the derivation and thorough analysis of the It\^{o} equation on $E_\eps(u_\eps(t))$. The noise brings into the system an amount of energy that is not zero on the average. Therefore, the Ginzburg-Landau energy of the solution to~\eqref{SGL_Strat} does not decrease with time. This is new, compared to the unperturbed flows. Still, we are able to establish a probabilistic counterpart of \eqref{Intro:EnergyWithVortices}. With it, the tightness of $(J(u_{\eps}))_{\eps}$ and $(\mu_{\eps}(u_{\eps}))_{\eps}$ follows from the deterministic Ginzburg-Landau theory. This part of our proof is more straightforward than the corresponding part in~\cite{RW}. Our computations are plainer because they are done for a simpler noise. Our arguments for the tightness are softer because we work in functional spaces that do not involve time. 

The proof of Theorem~\ref{structure} is the most novel and the most subtle part of our work. The analogy with the Allen-Cahn case is not helpful here. To characterize the limiting distribution $\mathcal P_{J}$, we should reproduce the analytic proof leading to \eqref{Intro:ConvergenceOfJacobian}. In $(J(u_{\eps_{n}}(t)))_{\eps_{n}}$, we should consider the limit $\eps_{n}\to 0$ \textit{pointwise} with respect to the elements of the underlying probability space. By Theorem~\ref{tightness}, the sequence $(J(u_{\eps_{n}}(t)))_{\eps_{n}}$ converges only weakly. In order to get some pointwise convergence, we apply the Skorokhod representation theorem. As mentioned before, we can not do that in the original space, since $(C_{0}^{0,\alpha})^{*}$ is not separable. Our switching to a larger space $W^{-1,q}$ is therefore a technical matter. We obtain a new sequence $(\tilde J_{n})$ that converges almost surely on a new probability space. The distribution of $\tilde J_n$ on $W^{-1,q}$ is the same as that of $J(u_{\eps_n}(t))$, for every $n$. However, the introduction of $(\tilde J_{n})$ gives rise to a further problem. There is no guarantee that $(\tilde J_{n})$ is a sequence of \textit{Jacobians} generated by a sequence of complex functions. This is related to the fact that we can not restore a function from its Jacobian: The mapping~$u\mapsto \det \nabla u$ is not invertible. Therefore we can not apply the result of Jerrard and Soner to the sequence $(\tilde J_{n})$. We can not assure that the limit of $(\tilde J_{n})$ is a sum of weighted delta measures. A problem of this kind does not arise for the stochastic Allen-Cahn equation. To characterize the singular set of the solution, one uses an invertible function of $U_{\eps}$.  

We resolve our problem by working with the product sequence $((J(u_{\eps_{n}}(t)), u_{\eps_{n}}(t)))_{\eps_{n}}$. The energy estimates yield an uniform bound on $(u_{\eps}(t))_{\eps}$ in $L^{2}(D)$. Thus, the family of the solutions $(u_{\eps_{n}}(t))_{\eps_{n}}$ is tight on any negative order Sobolev space. In the representing sequence~$((\tilde J_{n}, \tilde u_{n}))$, we have that $\tilde J_{n}=\det\nabla \tilde u_{n}$ almost surely. This finally allows us to apply the deterministic result of~\cite{JerrardSoner} to $(\tilde J_{n})$ in a pointwise manner. From the analytic point of view, the compactness properties of $(u_{\eps}(t))_{\eps}$ are extremely weak. It is therefore especially remarkable that they turn out to be useful in the stochastic framework. 

We are aware of no other strictly mathematical work that studies the stochastic Ginz\-burg-Landau vortices. What regards the subject, the most relevant for us are the numerical simulations by Deang, Du, and Gunzburger~\cite{deang2001stochastic,GunzburgerCompPhys}. The authors are also concerned with the formation of the Ginzburg-Landau vortices in the presence of the noise. On the other hand, there are purely mathematical works by Barton-Smith~\cite{barton2004global} and Kuksin and Shirikyan~\cite{kuksin2004randomly}. 
These authors focus on the long-time behavior of the solution rather than on the vortices. They consider the variants of the stochastic Ginzburg-Landau that are similar to the perturbed mixed flow. We point out that, in all the works mentioned above, the noise is not of a convective form. Moreover, their settings always correspond to the situation of a fixed positive $\eps$. 

\section{Preliminaries}\label{preliminaries}
\subsection{Notation}

Everywhere in the text $C$ denotes a constant that is independent of $\eps$.

We sum over repeated indices.

For a random variable $\xi$, $\mathcal L(\xi)$ is the law of $\xi$.

We view the unknown function $u_{\eps}$ sometimes as complex-valued and sometimes as $\R^{2}$-valued. For complex numbers $u, v\in \C\simeq\R^{2}$, $(u, v):=u^{1}v^{1}+u^{2}v^{2}$ is the real scalar product.

The matrix $(\nabla u\otimes \nabla u)$ with the entries 
$(\nabla u\otimes \nabla u )_{jk}=(\partial_{j} u,\partial_{k} u)$ is the \textit{stress tensor} of $u$. 

We denote by 
\begin{equation}\label{rhs}
f_\eps(u):=\Delta u+(1/\eps^2)(1-|u|^2)u
\end{equation}
the negative $L^2$-gradient of the Ginzburg-Landau energy \eqref{GLEnergy}. The drift term in~\eqref{SGL_Strat} is equal to $\log(1/\eps)f_{\eps}(u_{\eps})$. 

The \textit{energy density} is given by 
\begin{equation}\label{eq:EnergyDensity}
e_{\eps}(u)=\frac{1}{2}|\nabla u|^2+\frac{1}{4\eps^2}(1-|u|^2)^2,
\end{equation}
this is the integrand in the Ginzburg-Landau energy functional \eqref{GLEnergy}. Recall that the rescaled energy density~\eqref{eq:definitionOfRescaledEnergyDensity} is given by $\mu_{\eps}(u)=k_{\eps}e_{\eps}(u)$, with  
\[k_{\eps}:=(\log(1/\eps))^{-1}.\]
In the same vein, the \textit{rescaled energy functional} is given by 
\begin{equation}\label{eq:RescaledEnergyDef}
\mathcal{E}_\eps(u):=k_\eps E_\eps(u).
\end{equation}
If $u_{\eps}(t)$ is the solution of \eqref{SGL_Strat}, we write $\mathcal{E}_\eps(t)$ instead of $\mathcal{E}_\eps(u_{\eps}(t))$.
Similarly, $\mu_{\eps}(t)$ is the rescaled energy density~$\mu_{\eps}(u_{\eps}(t))$ associated to the solution of \eqref{SGL_Strat}. We use the same symbol for the measure on $D$ defined by~$\mu_{\eps}(t)$. For a Borel set $B\subset D$, we set
\[\mu_{\eps}(t)(B):=
k_{\eps}\int\limits_{B}e_{\eps}(u_{\eps})\,dx.
\] 
Note that \[\mu_{\eps}(t)(D)=\lVert \mu_{\eps}(t)\rVert_{L^{1}}=\mathcal{E}_\eps(t).\] 

We denote by $\nnu$ the outer unit normal on the boundary of $D$. We write $D_T$ for the space-time cylinder $D\times[0,T]$ with a fixed $T>0$. The set $\overbar{D_{T}}$ is the closure of $D_{T}$.

The Sobolev space $H^{1}(D)$ consists of functions in $L^{2}(D)$ that have weak gradients again belonging to $L^{2}(D)$. The space $C^{0}(D)$ is the space of functions that are continuous on $D$. The subspace $C_{0}^{0}(D)\subset C^{0}(D)$ consists of functions that are zero on the boundary of $D$.
The space $C_t^\alpha C_x^{2\alpha}(\overbar{D_T})$, with $\alpha\in(0, 1]$, is the Banach space of functions that are continuous on~$ \overbar{D_T}$ and satisfy 
\[[f]_{\alpha}:=\sup\limits_{(x,t)\neq(y,s)}\frac{|f(x,t)-f(y,s)|}{|x-y|^{2\alpha}+|s-t|^\alpha}<\infty.
\]
The corresponding norm is given by
\[\lVert f\rVert_{C_t^\alpha C_x^{2\alpha}(\overbar{D_T})}:=\lVert f\rVert_{C^{0}(\overbar{D_T})}+[f]_{\alpha}.
\] 
The spaces $C_t^{k,\alpha}C_x^{l,\beta}(\overbar{D_T})$, $C^{k,\alpha}(\overbar D)$, and $C^{k,\alpha}_{0}(D)$ with $k, l\in \N$ and $\alpha, \beta\in (0,1]$ are defined in the same manner. The space $C^{k,\alpha}_{c}(D)$ consists of functions in $C^{k,\alpha}(\overbar D)$ that have compact support in $D$.

The space $(C_0^{0,\alpha}(D))^*$ is the dual of the space of real-valued $\alpha$-H\"{o}lder continuous functions on $D$ that are zero on $\partial D$. Here we take $\alpha\in(0, 1]$. For $a$ and $b$ in~$D$, the $(C_0^{0,1}(D))^*$-norm of 
$ \delta_{a}-\delta_{b}$ roughly corresponds to the Euclidian distance between $a$ and~$b$.

The space $W^{-1,q}(D)$ is the dual of the space $W^{1,p}_{0}(D)$ of trace-zero Sobolev functions. Here, $1/p+1/q=1$. 

For the spaces of functions defined on $D$, we shall suppress the dependence on the domain in the notation. Accordingly, we shall write $(C_0^{0,\alpha})^*$\!, $L^{p}$\!, and so on. The functions in these spaces can be either $\C$ or $\R$-valued, depending on the context.

\subsection{Solution}\label{sec:solution}

We assume that the external vector field $\F(x)$ in~\eqref{SGL_Strat} belongs to the space~$C_{c}^{3,\sigma}(D;\R^{2})$ with~$\sigma\in(0,1]$.

We complement~\eqref{SGL_Strat} with either Dirichlet 
\begin{equation}\label{Dirichlet}
u_{\eps}(x,t)|_{\partial D}=g(x)
\end{equation}
or zero Neumann 
\begin{equation}\label{Neumann}
\partial_{\nnu}u_{\eps}(x,t)|_{\partial D}=0 
\end{equation}
boundary conditions. In~\eqref{Dirichlet}, the deterministic function $g\in C^{\infty}(\partial D, \sS^{1})$ has a non-zero winding number. This gives rise to a topological constraint on the solution. Any $H^{1}$-function $u:D\to\C$ that coincides with $g$ on $\partial D$ \textit{must} have vortices in the interior of~$D$. In the Neumann case, we require that the initial data~$u_{\eps}^{0}$ has vortices. 

In both cases, we furthermore assume that the initial data 
\begin{equation}\label{InitialData}
u_{\eps}(x,0)=u_{\eps}^{0}(x)
\end{equation}
for \eqref{SGL_Strat} is deterministic and smooth ($C^{\infty}(\overbar{D})$). We assume in addition that $|u_{\eps}^{0}(x)|\leqslant 1$ in~$D$. As a consequence of~$u_{\eps}^{0}$ having vortices~\cite{SandierBallConstr}, there exists a constant $C_{0}$ such that 
\begin{equation*}
E_{\eps}(u_{\eps}^{0})\geqslant C_{0}\log(1/\eps).
\end{equation*}
We assume that the opposite inequality 
\begin{equation}\label{eq:EnergyAtTimeZero}
E_{\eps}(u_{\eps}^{0})\leqslant C\log(1/\eps)
\end{equation}
holds for $u_{\eps}^{0}$ as well. The coefficient $C$ can be strictly larger than $C_{0}$. 

An initial data that satisfies all our requirements exists for any given vortex configuration~$(a_{1},..., a_{N}; d_{1}, ..., d_{N})$, cf. \cite{JerrardSpirnRefined}, Lemma 14 and \cite{JS_Dynamics}, Remark 2.1. We can consider initial data of a special form
\[u_{\eps}^{0}(x)=u_{*}(x; a_{1},..., a_{N}; d_{1}, ..., d_{N})\cdot \prod_{k=1}^{N}f(|x-a_{k}|/\eps). \]
Here, the function $u_{*}(x; a_{1},..., a_{N}; d_{1}, ..., d_{N})$ is the canonical harmonic map~\cite{BBH} associated to the vortex configuration. This is a uniquely determined $\mathbb {S}^{1}$-valued map that satisfies the correct boundary condition, has vortices of degrees~$d_{k}$ at positions~$a_{k}$ and belongs to $C^{\infty}(D\setminus\{a_{1},..., a_{N}\})$.
The real-valued function $f$ is an arbitrary $C^{\infty}$ monotone non-decreasing function with $f(0)=0$, $f(r)=1$ for~$r\geqslant (1/2)\cdot \min_{j\neq k}\{|a_{j}-a_{k}|, \dist(a_{k}, \partial D)\}$. Note that we do not need such a smooth initial data for the proof of Theorem~\ref{existence} -- a~$u_{\eps}^{0}\in C^{3, \theta}(\overbar{D})$ with $\theta\in(0, \sigma)$ would be sufficient. 

We now recall the definition of a stochastically strong solution. 
\begin{mydef}\label{def:solution}
The initial-boundary value problems \eqref{SGL_Strat}-\eqref{InitialData}-\eqref{Dirichlet} and \eqref{SGL_Strat}-\eqref{InitialData}-\eqref{Neumann} have a strong solution on a time-interval $[0,T]$ if, for every filtered probability space $(\Omega, \mathfrak{F},\{\mathfrak{F}_{t}\}, \mathbb{P})$, $t\in [0, T]$ with a right-continuous complete filtration and every standard one-dimensional Brownian motion $B_{t}$ adapted to this filtration, there exists a process~$u_\eps(x,t,\omega): \overbar{D}\times[0,T]\times\Omega\rightarrow \C$ such that
\begin{itemize} 
\item for some $\beta\in(0,1)$, $u_\eps$ is a continuous $C^{3,\beta}(\overbar{D})$-valued semimartingale adapted to~$\{\mathfrak{F}_{t}\}$;  
\item $u_\eps(\cdot,t,\omega)$ satisfies the boundary conditions \eqref{Dirichlet} or \eqref{Neumann} for all $t\in[0,T]$ and almost all $\omega\in \Omega$;
\item  for all $x\in D$ and $t\in[0,T]$, the equality 
\begin{align}\label{eq:SGLIntegralForm}
u_\eps(x,t, \omega)=u^0_\eps(x)+\log(1/\eps) \int\limits_0^t \Delta u_{\eps}(x,s, \omega)&+\frac{1}{\eps^{2}}(1-|u_{\eps}|^{2})u_{\eps}(x,s,\omega)\,ds\nonumber\\
&+\int\limits_0^t(\F\cdot\nabla)u_{\eps}(x, s,\omega) \circ dB_{s}
\end{align}
holds $\pP$-almost surely.
\end{itemize}
\end{mydef}

The logarithmic coefficient at the drift in~\eqref{SGL_Strat} and \eqref{eq:SGLIntegralForm} is inherited from the deterministic parabolic Ginzburg-Landau equation. It singles out the correct time scale on which the vortices move.

We need the It\^{o} equivalent of \eqref{SGL_Strat} for the further analysis in Section~\ref{sec:energy}. The equation~\eqref{SGL_Strat} in the Stratonovich formulation corresponds to the following equation in the It\^{o} formulation (\cite{Kunita}, Chapter~6)
\begin{equation}\label{SGL_I}
 d u_{\eps}=\log(1/\eps)(\Delta u_{\eps}+(1/\eps^{2})(1-|u_{\eps}|^{2})u_{\eps})\,dt+(1/2)(\F\cdot\nabla)^{2} u_\eps\,dt+
(\F\cdot\nabla) u_{\eps}\,dB_{t}.
\end{equation}
Here, the Stratonovich-It\^{o} correction term is given by
\[(\F\cdot\nabla)^{2} u_\eps=F^{j}F^{k}\partial^{2}_{jk}u_{\eps}+F^{j}\partial_{j}F^{k}\partial_{k}u_{\eps}.\]
Hence, the transformation changes the leading-order term in the drift. 

\begin{remark}
In the It\^o equation~\eqref{SGL_I}, the drift contains an elliptic operator of the second order. The diffusion depends on a differential operator of the first order. A solution to an equation of this form exists only if the stochastic parabolicity condition~\cite{flandoli1996stochastic} is satisfied. Here, it is satisfied automatically. Indeed, \eqref{SGL_I} stems from a Stratonovich equation in which the drift and the diffusion are correctly balanced. 
\end{remark}

\subsection{Stochastic flows}\label{sec:StochasticFlows}

We recall Kunita's concept of the stochastic flows on $D$. For an extensive discussion, see the monograph of Kunita~\cite{Kunita}, Chapter~3 and Sections 2, 5, and~6 in Chapter~4.

The Stratonovich flow $\varphi_{s,t}$ corresponding to $\F$ is defined as follows. Let $T$, $(\Omega, \mathfrak{F}, \{\mathfrak F_{t}\}\mathbb{P})$, and $B_{t}$ be as in Definition~\ref{def:solution}. For a fixed $s\in [0,T)$, the process $\varphi_{s,t}(x)$ is, for every~$x\in D$ and $t\in(s,T]$, the solution of the stochastic differential equation
\[\left\{
\begin{array}{l}
d\,\varphi_{s,t}(x)=- \F(\varphi_{s,t}(x))\circ dB_t,\\
\varphi_{s,s}(x)=x.
\end{array}
\right.
\]
This means that for every $x\in D$ and $s\in[0,t)$, the process $\varphi_{s,t}(x)$ is $\mathfrak F_{t}$-adapted. 
Since $\F$ belongs to the space~$C_{c}^{3,\sigma}$ with $\sigma>0$, the family~$\varphi_{s,t}$ is, almost surely, a two-parameter family of  $C^{3,\theta}$-diffeomorphisms of $\overbar D$, for any $\theta<\sigma$. For the rest of the paper, we fix some~$\theta\in(0, \sigma)$. The transformations~$\varphi_{s,t}$ act trivially on the boundary of $D$. Furthermore, they satisfy the flow property: for any~$r\in(s,t)$, there holds
\[\varphi_{s,t}=\varphi_{r,t}\circ \varphi_{s,r}.
\]  
\begin{remark}
In our case, $\varphi_{s,t}(x)$ is explicitly given by 
$\varphi_{s,t}(x)=\Phi(x, B_t-B_s).
$
Here, $\Phi$ is the deterministic flow generated by $\F$. For every $x\in D$ and every $t\in\R$, $\Phi(x, t)$ solves
\[
\left\{
\begin{array}{l}
\partial_t\Phi(x,t)=- \F(\Phi(x,t)),\\
\Phi(x,0)=x.
\end{array}
\right.
\]
\end{remark}

\section{Existence and uniqueness of the solution}\label{sec:existence}
In this section, we prove Theorem~\ref{existence}. The arguments for Dirichlet and Neumann boundary conditions are in most aspects identical. We thus focus on the Dirichlet case.
\begin{proof}[Proof of Theorem~\ref{existence}]
We fix a filtered probability space $(\Omega, \mathfrak{F},\{\mathfrak{F}_{t}\},\mathbb{P})$ with $t\in [0,T]$ and a Brownian motion $B_{t}$ that satisfy the requirements of Definition \ref{def:solution}. We study \eqref{SGL_Strat} on this space for a fixed $\eps\in(0,1)$. 

We transform the equation into a family of random PDEs.
Let $\varphi_{s,t}(x)$  be the Stra\-to\-no\-vich flow generated by~$\F$, as described in Section \ref{sec:StochasticFlows}. The flow defines a random change of variables.  Note that the space $C^{3,\theta}(\overbar{D}; \overbar{D})$ is invariant under this transformation. We define the function 
 \begin{equation}\label{eq:changeOfVariables}
 v_{\eps}(x,t):=u_{\eps}(\varphi_{0,t}(x),t)
 \end{equation}
 in the new variables. 

An explicit computation shows that for $k\in\{1, 2\}$
\[ dv^k_{\eps}(x,t)=du^k_{\eps}(\varphi_{0,t}(x),t)-(\F\cdot\nabla) u^k_{\eps}(\varphi_{0,t}(x),t) \circ dB_t
\]
and 
\begin{align*}
f^{k}_{\eps}(v_{\eps}(x,t))&=f^{k}_\eps(u_\eps(\varphi_{0,t}(x),t))\\
&=a^{ij}(x,t)\partial^2_{ij} v_\eps^k(x,t)+b^{i}(x,t)\partial_i v_\eps^k(x,t)+
(1/\eps^2)(1-|v_\eps(x,t)|^2)v_\eps^k(x,t).
\end{align*}
The coefficients $a^{ij}$ and $b^{i}$ depend on the flow $\varphi_{s,t}$ and are thus random. They are given by
\[
a^{ij}(x,t)=\partial_l ((\varphi_{0,t})^{-1})^i(\varphi_{0,t}(x),t)\cdot\partial_l ((\varphi_{0,t})^{-1})^j(\varphi_{0,t}(x),t) 
\]
 and 
\[b^{i}(x,t)= \partial^2_{ll} ((\varphi_{0,t})^{-1})^i(\varphi_{0,t}(x),t)\]
with $i,j,l\in\{1,2\}$.

The process $u_{\eps}(x,t,\omega)$ is the unique solution of~\eqref{SGL_Strat}-\eqref{InitialData}-\eqref{Dirichlet} if and only if $v_{\eps}(x,t,\omega)$ with $v_{\eps}\in C^{1,\beta}_{t}C_{x}^{3,2\beta}$ for some $\beta>0$ is the unique solution of
\begin{equation}\label{v_eq}
\left\{
\begin{array}{l}
 k_\eps\partial_{t} v_{\eps}=a^{ij}(x,t,\omega)\partial_{ij}^2 v_\eps+b^{i}(x,t,\omega)\partial_i v_\eps+
(1/\eps^2)(1-|v_\eps|^2)v_\eps,\\
v_\eps|_{\partial D}=g,\\
v_\eps(0,x)=u_\eps^0(x).
\end{array}
\right.
\end{equation}
Indeed (cf.~\cite{Kunita}, Lemma~6.2.3), if $u_{\eps}$ solves \eqref{SGL_Strat}-\eqref{InitialData}-\eqref{Dirichlet}, then, due to definition~\eqref{eq:changeOfVariables} of~$v_{\eps}$ and the computations following it, $v_{\eps}$ solves~\eqref{v_eq}. Conversely, if $v_{\eps}\in C^{1,\beta}_{t}C_{x}^{3,2\beta}$ solves~\eqref{v_eq} then, again by~\eqref{eq:changeOfVariables} and the chain rule, $u_{\eps}$ satisfies~\eqref{SGL_Strat} and \eqref{InitialData}-\eqref{Dirichlet}. Note that $v_{\eps}$ depends on $\omega$ only because the coefficients in~\eqref{v_eq} do depend on it.

We now prove that~\eqref{v_eq} has a unique solution in the space~$C^{1,\beta}_t C^{3,2\beta}_x( \overbar{D_T})$, almost surely. Our argument relies on the properties of $a^{ij}$ and $b^{i}$ that are satisfied almost surely. Therefore, it is valid for almost every $\omega\in\Omega$. We apply it pointwise in $\omega$. The reasoning is rather standard and rather lengthy. We just sketch it below and refer the interested reader to Section~4.3 in~\cite{RWTHDiss} for a complete proof.

We interpret \eqref{v_eq} as a system of two equations on the real and the imaginary part of~$v_{\eps}$. Due to the definition of $a^{ij}$, this is a strictly parabolic system. The coefficients $a^{ij}(x,t)$ and $b^{i}(x,t)$ belong, for every $t\in[0,T]$, to the spaces~$C^{2,\theta}(\overbar{D})$ and $C^{1,\theta}(\overbar{D})$, respectively. Furthermore, they are H\"{o}lder continuous in time with any exponent 
below~$1/2$. These properties follow from the regularity of $\varphi_{s,t}$. Consequently, 
$a^{ij}$, $b^{i}\in C_t^\beta C_x^{2\beta}(\overbar{D_T})$ for any~$\beta<\theta/2$.
If $x$ belongs to $D\setminus\supp \F$, we have that $a^{ij}(x,t)=\delta^{ij}$ and  $b^{i}(x,t)=0$ for every~$t\in[0,T]$. 

First, we find a solution of~\eqref{v_eq} in the space~$C^{1,\beta}_t C^{2,2\beta}_x(\overbar{D_T})$ for some $\beta\in(0,\theta/2)$. To this end, we apply the Leray-Schauder fixed point theorem in the manner discussed in~\cite{Ladyzhenskaya_Parabolic}, Section~V.6 (see also~\cite{GilbargTrudinger}, Section~11.4 and \cite{ZeidlerI}, Section~6.8). Even though~\cite{Ladyzhenskaya_Parabolic} treats the case of a real-valued unknown function, the method is also applicable to~\eqref{v_eq}. Indeed, the two equations in~\eqref{v_eq} are coupled only through a (polynomial) nonlinearity. We consider the family of mappings~$\mathcal {F}_{\lambda}$, $\lambda\in[0,1]$ of $C^{1,\beta}_t C^{2,2\beta}_x(\overbar{D_T})$ into itself, where~$\mathcal{F}_{\lambda}(v_{\eps})=w$ if $w$ solves the linear problem corresponding to~\eqref{v_eq}
\[
 k_\eps\partial_{t} w=\lambda(a^{ij}(x,t,\omega)\partial_{ij}^2 w+b^{i}(x,t,\omega)\partial_i w)+(1-\lambda)\Delta w+N(v_{\eps}, u_{\eps}^{0})
\]
with zero initial and boundary conditions. The function $N(v_{\eps}, u_{\eps}^{0})$ is nonlinear in $v_{\eps}$, but does not involve derivatives of~$v_{\eps}$. We verify that $\mathcal {F}_{\lambda}$ satisfies the conditions of the Leray-Schauder fixed point theorem. In doing so, we use the Schauder estimates for $w$ (\cite{Ladyzhenskaya_Parabolic}, Theorem~IV.5.5) and an estimate on the $L^{\infty}$-norm of~$|v_{\eps}|^{2}$. The latter follows from the weak maximum principle applied to the equation on $(1-|v_{\eps}|^2)$. We thus obtain a solution to~\eqref{v_eq} as a fixed point of the mapping $\mathcal{F}_{1}$. Then, we directly check that the solution is unique. Finally, we show that~$v_{\eps}$ has a higher regularity. We apply Theorem~IV.5.5 of~\cite{Ladyzhenskaya_Parabolic} to the equations on the spatial derivatives of~$v_\eps$; these new equations are already linear in the unknown function. In this way, we get a solution that belongs to the space~$C^{1,\beta}_t C^{3,2\beta}_x( \overbar{D_T})$. 

Now we have that
\[u_\eps(x,t)=v_\eps((\varphi_{0,t})^{-1}(x),t),
\]
with $v_\eps\in C^{1,\beta}_t C^{3,2\beta}_x( \overbar{D_T})$ and $(\varphi_{0,t})^{-1}\in C^{\beta}_t C^{3,\theta}_x( \overbar{D_T})$. Hence, the regularity of $u_\eps$ is the minimal of the two. We conclude that $u_\eps$ belongs to the space $C^{\beta}_t C_{x}^{3,2\beta}(\overbar{D_T})$, as claimed.
\qed
\end{proof}
\begin{remark}
The difference between the Dirichlet and the Neumann cases consists only in the method of obtaining the $L^{\infty}$-estimate on $|v_{\eps}|^{2}$.
In the Neumann case, the maximum principle is not at hand. Instead, we use the trick from Section 3 of~\cite{Alouges_Soyeur}. In both cases, we are able to conclude that $|v_{\eps}|^2\leqslant 1$ in $\overbar {D_{T}}$. 
\end{remark}

\begin{remark}
The assumption that the field $\F$ is compactly supported in $D$ is essential for the proof. Otherwise the flow~$\varphi_{s,t}$ would act non-trivially on the boundary $\partial D$. Then, the equation~\eqref{v_eq} would be posed in a domain that itself evolves with time.    
\end{remark}

\section{Ginzburg-Landau energy in the stochastic case}\label{sec:energy}
In this section, we derive the equation on the Ginzburg-Landau energy of the solution to \eqref{SGL_Strat}. With its help, we establish the control on the energy growth in Proposition~\ref{EnergyEstimate}. Furthermore, we discuss some consequences of the obtained estimates.
\subsection{The equation for $E_\eps(u_\eps(t))$}
The calculations and notation in this subsection closely follow those of Section 6 in~\cite{RW}. 

\begin{proposition}\label{Energy_dynamics}
If $u_\eps$ is the solution of (\ref{SGL_Strat}), its Ginzburg-Landau energy $E_{\eps}(u_\eps)$ satisfies the equation
 
\begin{align}\label{Energy_Evolution}
E_{\eps}(u_\eps(t_{2}))-&E_{\eps}(u_\eps(t_{1}))= - \log(1/\eps)\int\limits_{t_1}^{t_2}\int\limits_D |f_\eps(u_\eps)|^2\,dxd\tau
-\int\limits_{t_1}^{t_2}\int\limits_D (f_\eps(u_\eps), (\F\cdot\nabla) u_\eps)\,dx dB_\tau\nonumber\\
&+\frac{1}{2}\int\limits_{t_1}^{t_2}\int\limits_D e_\eps(u_\eps)\diverg(\F\cdot \diverg \F)\, dxd\tau+\frac{1}{2}\int\limits_{t_1}^{t_2}\int\limits_D|\nabla u_\eps\cdot\nabla \F|^2\, dxd\tau\nonumber\\
&+\int\limits_{t_1}^{t_2}\int\limits_D \nabla u_\eps:(\nabla u_\eps\cdot(\nabla \F\cdot \nabla \F))\,dxd\tau
-\int\limits_{t_1}^{t_2}\int\limits_D \nabla u_\eps:(\nabla u_\eps\cdot\nabla \F )\cdot\diverg \F\,dxd\tau\nonumber\\
&-\frac{1}{2}\int\limits_{t_1}^{t_2}\int\limits_D\nabla u_\eps:(\nabla u_\eps\cdot\nabla(\nabla \F\cdot \F))\,dxd\tau
\end{align}
  
for all $0\leqslant t_{1}<t_{2}\leqslant T$, $\pP$-almost surely. 
\end{proposition}
We recall that the quantity $f_\eps(u_\eps)$ has been defined in~\eqref{rhs}, and $e_{\eps}(u_{\eps})$ in~\eqref{eq:EnergyDensity}.

We introduce a scalar function 
\[\psi(x):=\frac{1}{2}\diverg(\F\cdot \diverg \F)
\]
and a matrix-valued function
\[\Psi(x):=\frac{1}{2}\nabla \F\cdot\nabla \F^\top+(\nabla \F\cdot\nabla \F)^\top-\, \diverg \F\cdot\nabla \F^\top-\frac{1}{2}(\nabla(\nabla \F\cdot \F))^\top.
\]
For a matrix $A$, $A^{\top}$ is the transpose matrix with the entries $(A^{\top})_{ij}=A_{ji}$. For two matrices $A$ and $B$, $A\cdot B$ is the usual matrix product, which is again a matrix. Their Frobenius product $A:B=\trace(A\cdot B^\top)$ is a scalar. 

Note that $\psi(x)$ and $\Psi(x)$ depend only on the field $\F$ and its derivatives up to the second order.
We write the quantity from the last three lines of \eqref{Energy_Evolution} in a more compact form as
\begin{equation}\label{Extra}
\int\limits_{t_1}^{t_2}\int\limits_D e_\eps(u_\eps)\psi(x)+\trace(\nabla u_\eps\cdot\Psi(x)\cdot\nabla u_\eps^\top)\,dxd\tau.
\end{equation}
It appears in~\eqref{Energy_Evolution} due to the Stratonovich-It\^{o} correction in~\eqref{SGL_I} and the application of the It\^{o} lemma. It thus accounts for the stochastic effects.

The proof is a direct application of the It\^{o} lemma to $E_\eps(u_\eps(t))$. At first glance, the equation on $E_\eps(u_\eps(t))$ should contain terms with the mixed second-order partial derivatives of~$u_\eps$. These quantities are beyond our control. According to Proposition~\ref{Energy_dynamics}, they are eventually eliminated from the equation. This is possible because the diffusion term is a linear function of $\nabla u_{\eps}$ and the energy contains the quantity $|\nabla u_{\eps}|^{2}$. The troublesome quantities appear in the Stratonovich-It\^{o} correction term, which depends on the expression $(\F\cdot\nabla)u_\eps$, and in the second-order derivative of $E_\eps(u_{\eps})$. In a sum, they can be exactly balanced.

\begin{proof}[Proof of Proposition~\ref{Energy_dynamics}.]
We are going to apply the It\^{o} lemma to $E_{\eps}(u_\eps(t))$. Therefore, we need the explicit expression for the derivatives of $E_\eps(u)$. The first and the second Fr\'e\-chet derivatives of $E_\eps(u)$ with respect to the Sobolev space $H^{1}$ are given by
\begin{equation}\label{eq:firstDerOfEn}
DE_\eps(u)\langle v \rangle=\int\limits_D (\nabla u:\nabla v)-(1/\eps^{2})(1-|u|^{2})(u,v)\,dx
\end{equation}
and 
\begin{equation}\label{eq:secondDerOfEn}
D^2E_\eps(u)\langle v,w\rangle =
\int\limits_D(\nabla w:\nabla v)-(1/\eps^{2})(1-|u|^{2})(w,v)+(2/\eps^{2})(u,w)(u,v)\,dx.
\end{equation}
Here, the functions $v$ and $w$ are arbitrary elements of $H^{1}$\!. Due to the Sobolev embedding~$H^{1}\hookrightarrow L^{p}$ that holds for every $p\in[1,\infty)$ and to the form of the nonlinearity in~$E_{\eps}(u)$, the mappings $E_{\eps}$, $DE_{\eps}$, and 
$D^{2}E_{\eps}$ are uniformly continuous on bounded subsets of~$H^{1}$. Therefore, we obtain with the It\^{o} formula (\cite{daPrato_Zabczyk}, Theorem 4.17) that $E_{\eps}(u_\eps(t))$ satisfies the equation
\begin{align}\label{Energy_first}
E_{\eps}(u_\eps(t_{2}))-& E_{\eps}(u_\eps(t_{1})) = \log(1/\eps)\int\limits_{t_1}^{t_2}DE_\eps(u_\eps)\langle\Delta u_{\eps}+(1/\eps^{2})(1-|u_{\eps}|^{2})u_{\eps}\rangle \,d\tau\nonumber\\
&+\int\limits_{t_1}^{t_2}DE_\eps(u_\eps)\langle (\F\cdot\nabla)  u_{\eps}\,dB_\tau\rangle 
+\frac{1}{2}\int\limits_{t_1}^{t_2}DE_\eps(u_\eps)\langle(\F\cdot\nabla)^{2} u_\eps \rangle \,d\tau\nonumber\\
&+\frac{1}{2}\int\limits_{t_1}^{t_2}D^2 E_\eps(u_\eps)\langle(\F\cdot\nabla)u_{\eps}\,dB_\tau,(\F\cdot\nabla)  u_{\eps} \,dB_\tau\rangle,
\end{align}
$\pP$-almost surely for any $0\leqslant t_{1}<t_{2}\leqslant T$.

We denote the seven terms on the right-hand side of~\eqref{Energy_Evolution} by $S_{1}, ... , S_{7}$. Our aim is to transform the right-hand side of~\eqref{Energy_first} into the sum $S_{1}+...+S_{7}$. With~\eqref{eq:firstDerOfEn} and~\eqref{eq:secondDerOfEn}, we see immediately that the the first and the second term on the right-hand side of \eqref{Energy_first} are equal to $S_{1}$ and $S_{2}$, respectively.

We transform the remaining part of~\eqref{Energy_first} into
\begin{multline*}\frac{1}{2}\int\limits_{t_1}^{t_2}DE_\eps(u_\eps)\langle(\F\cdot\nabla)^{2} u_\eps \rangle \,d\tau
=-\frac{1}{2}\int\limits_{t_1}^{t_2}\int\limits_D (\Delta u_{\eps}+(1/\eps^{2})(1-|u_{\eps}|^{2})u_{\eps},(\F\cdot\nabla)^{2} u_\eps )\,dxd\tau\\
=-\frac{1}{2}\int\limits_{t_1}^{t_2}\int\limits_D \left(\Delta u_{\eps},(\F\cdot\nabla)^{2}u_\eps \right)\,dxd\tau-\frac{1}{2}\int\limits_{t_1}^{t_2}\int\limits_D((1/\eps^{2})(1-|u_{\eps}|^{2})u_{\eps},(\F\cdot\nabla)^{2} u_\eps)\,dxd\tau\\
=: T_1+T_2
\end{multline*}
 and
\begin{multline*}
\frac{1}{2}\int\limits_{t_1}^{t_2}D^2 E_\eps(u_\eps)\langle(\F\cdot\nabla) u_{\eps}\, dB_\tau, (\F\cdot\nabla) u_{\eps}\, dB_\tau\rangle\\
=\,\frac{1}{2}\int\limits_{t_1}^{t_2}\int\limits_D \nabla((\F\cdot\nabla) u_\eps):\nabla((\F\cdot\nabla) u_\eps) \,dxd\tau
+\frac{1}{2} \int\limits_{t_1}^{t_2}\int\limits_D (2/\eps^2)((u_\eps,(\F\cdot\nabla)u_\eps))^2\,dxd\tau\\
-\frac{1}{2} \int\limits_{t_1}^{t_2}\int\limits_D (1/\eps^2)(1-|u_\eps|^{2})|(\F\cdot\nabla) u_\eps|^2 \,dxd\tau
=: T_3+T_4+T_5.
\end{multline*}
It remains to bring the terms $T_1 - T_5$ to the desired form. Below, all our manipulations are made only in spatial variables. 

We consider first the terms $T_2, T_4, T_5$, which are related to the nonlinear part of the energy functional. We recall that the Stratonovich-It\^{o} correction term is explicitly given by
\[((\F\cdot\nabla)^{2} u_\eps)^j=\partial^2_{sl} u_\eps^j F^sF^l+ 
\partial_s u_\eps^j \partial_l F^sF^l.
\]
Therefore
\begin{multline*}
T_2=-\frac{1}{2}\int\limits_{t_1}^{t_2}\int\limits_D (1/\eps^{2})(1-|u_{\eps}|^{2})u_{\eps}^j\partial^2_{sl} u_\eps^j F^sF^l\,dxd\tau
\\-\frac{1}{2}\int\limits_{t_1}^{t_2}\int\limits_D (1/\eps^{2})(1-|u_{\eps}|^{2})u_{\eps}^j\partial_s u_\eps^j \partial_l F^sF^l\, \,dxd\tau
=T_{6}+T_{7}.
\end{multline*}
Note that
 $\partial_{s} (1-|u_\eps|^{2})^2= -4 (1-|u_\eps|^{2})u_\eps^{j}\partial_{s} u_\eps^{j}$ for $s=1,2$.

We work with the term $T_6$ and integrate by parts with respect to $x_l$. This gives
\begin{align*}
T_{6}=&-\frac{1}{2\eps^{2}}\int\limits_{t_1}^{t_2}\int\limits_D (1-|u_{\eps}|^{2})u_{\eps}^jF^sF^l\partial^2_{sl} u_{\eps}^j \,dxd\tau
=\frac{1}{2\eps^{2}}\int\limits_{t_1}^{t_2}\int\limits_D \partial_l((1-|u_{\eps}|^{2})u_{\eps}^jF^sF^l)\partial_s u_{\eps}^j \,dxd\tau\\
&=\frac{1}{2\eps^{2}}\int\limits_{t_1}^{t_2}\int\limits_D \left(-2 u_{\eps}\cdot\partial_l u_{\eps}\right)u_{\eps}^jF^sF^l \partial_s u_{\eps}^j \,dxd\tau
+\frac{1}{2\eps^{2}}\int\limits_{t_1}^{t_2}\int\limits_D (1-|u_{\eps}|^{2}) \partial_l u_{\eps}^j F^sF^l \partial_s u_{\eps}^j \,dxd\tau\\
& \ +\frac{1}{2\eps^{2}}\int\limits_{t_1}^{t_2}\int\limits_D (1-|u_{\eps}|^{2})u_{\eps}^j \partial_s u_{\eps}^j\partial_l(F^sF^l)\,dxd\tau\\
&=-T_5-T_4+\frac{1}{2\eps^{2}}\int\limits_{t_1}^{t_2}\int\limits_D (1-|u_{\eps}|^{2})u_{\eps}^j \partial_s u_{\eps}^j\partial_l(F^sF^l)\,dxd\tau.
\end{align*}

Thus, we have that
\begin{multline*}
T_2 + T_4 + T_5=-\frac{1}{8\eps^{2}}\int\limits_{t_1}^{t_2}\int\limits_D \partial_{s} (1-|u_\eps|^{2})^2 \cdot (\partial_l(F^sF^l)-\partial_l F^s F^l) \,dxd\tau\\
= \frac{1}{2}\int\limits_{t_1}^{t_2}\frac{1}{4\eps^{2}}\int\limits_D (1-|u_\eps|^{2})^2 \diverg\left(\F\cdot\diverg \F\right) \,dxd\tau.
\end{multline*}

We turn our attention to the terms $T_1$ and $T_3$, which contain the highest-order derivatives of $u_\eps$.

We write $T_1$ explicitly:
\[T_1=-\frac{1}{2}\int\limits_{t_1}^{t_2}\int\limits_D \partial^2_{kk} u_\eps^j\cdot\big( \partial^2_{sl} u_\eps^j F^sF^l+ 
\partial_l u_\eps^j \partial_s F^lF^s\big)\,dxd\tau=T_8+T_9.
\]

We integrate by parts with respect to $x_k$ in the term $T_9$ and see that
\begin{align*}
T_9&=-\frac{1}{2}\int\limits_{t_1}^{t_2}\int\limits_D\partial^2_{kk} u_\eps^j\partial_l u_\eps^j\partial_s F^l F^s \,dxd\tau\\
=&
\frac{1}{2}\int\limits_{t_1}^{t_2}\int\limits_D\partial_k u_\eps^j\partial^2_{kl} u_\eps^j\partial_s F^l F^s\, \,dxd\tau+
\frac{1}{2}\int\limits_{t_1}^{t_2}\int\limits_D\partial_k u_\eps^j\partial_l u_\eps^j\partial_k(\partial_s F^l F^s)\,dxd\tau\\
=&\frac{1}{2}\int\limits_{t_1}^{t_2}\int\limits_D\partial_l\big(\tfrac{1}{2}|\nabla u_\eps|^2\big)\partial_s F^l F^s \,dxd\tau+
\frac{1}{2}\int\limits_{t_1}^{t_2}\int\limits_D\partial_k u_\eps^j\partial_l u_\eps^j\partial_s F^l\partial_k F^s \,dxd\tau\\
& \ \ +\frac{1}{2}\int\limits_{t_1}^{t_2}\int\limits_D\partial_k u_\eps^j\partial_l u_\eps^j \partial^2_{s k} F^l F^s \,dxd\tau\\
=&-\frac{1}{2}\int\limits_{t_1}^{t_2}\int\limits_D\tfrac{1}{2}|\nabla u_\eps|^2(\diverg (\nabla \F\cdot \F))\,dxd\tau+\frac{1}{2}\int\limits_{t_1}^{t_2}\int\limits_D\nabla u_\eps:(\nabla u_\eps\cdot(\nabla \F\cdot \nabla \F))\,dxd\tau+T_{10}\\
=&\ R_{1}+\frac{1}{2}S_{5}+T_{10}.
\end{align*}
  
In the next step, we integrate by parts with respect to $x_k$ in the term $T_{10}$ and get
\begin{multline*}
T_{10}=-\frac{1}{2}\int\limits_{t_1}^{t_2}\int\limits_D\Delta u^j\partial_l u_\eps^j\partial_s F^l F^s \,dxd\tau-
\frac{1}{2}\int\limits_{t_1}^{t_2}\int\limits_D\partial_k u_\eps^j\partial^2_{l k} u_\eps^j  \partial_s F^l F^s \,dxd\tau\\
 -\frac{1}{2}\int\limits_{t_1}^{t_2}\int\limits_D\partial_k u_\eps^j\partial_l u_\eps^j\partial_s F^l\partial_k F^s \,dxd\tau\\
=-\frac{1}{2}\int\limits_{t_1}^{t_2}\int\limits_D f_\eps(u_\eps)^j\partial_l u_\eps^j\partial_s F^l F^s \,dxd\tau
+\frac{1}{2}\int\limits_{t_1}^{t_2}\int\limits_D(1/\eps^2)(1-|u_\eps|^2)u_\eps^j\partial_l u_\eps^j\partial_s F^l F^s \,dxd\tau\\
-\frac{1}{2}\int\limits_{t_1}^{t_2}\int\limits_D\partial_l\left(e_\eps(u_\eps)\right)\partial_s F^l F^s \,dxd\tau
+\frac{1}{2}\int\limits_{t_1}^{t_2}\int\limits_D\partial_l(\tfrac{1}{4\eps^2}(1-|u_\eps|^2)^2)\partial_s F^l F^s\,dxd\tau \\
-\frac{1}{2}\int\limits_{t_1}^{t_2}\int\limits_D\nabla u_\eps:(\nabla u_\eps\cdot(\nabla \F\cdot \nabla \F))\,dxd\tau.
\end{multline*}
The second and the fourth term above, both corresponding to the quantity $(1-|u_\eps|^2)^{2}$, cancel each other. We conclude that
\begin{multline*}
T_{10}=-\frac{1}{2}\int\limits_{t_1}^{t_2}\int\limits_D f_\eps(u_\eps)^j\partial_l u_\eps^j\partial_s F^l F^s \,dxd\tau 
-\frac{1}{2}\int\limits_{t_1}^{t_2}\int\limits_D\partial_l\left(e_\eps(u_\eps)\right)\partial_s F^l F^s \,dxd\tau \\
-\frac{1}{2}\int\limits_{t_1}^{t_2}\int\limits_D\nabla u_\eps:(\nabla u_\eps\cdot(\nabla \F\cdot \nabla \F))\,dxd\tau.
\end{multline*}
We note that 
$ f_\eps(u_\eps)^j\partial_l u_\eps^j=-\partial_l e_\eps(u_\eps)+\partial_k(\partial_l u_\eps^j\partial_k u_\eps^j).
$
Hence, 
\begin{multline*}
T_{10}=\frac{1}{2}\int\limits_{t_1}^{t_2}\int\limits_D \nabla u_\eps:(\nabla u_\eps\cdot\nabla (\nabla \F\cdot \F))\,dxd\tau 
-\frac{1}{2}\int\limits_{t_1}^{t_2}\int\limits_D\nabla u_\eps:(\nabla u_\eps\cdot(\nabla \F\cdot \nabla \F))\,dxd\tau \\
=-S_{7}-\frac{1}{2}S_{5}.
\end{multline*}

We write $T_3$ in more detail:
\begin{align*}
T_3 = & \ \frac{1}{2}\int\limits_{t_1}^{t_2}\int\limits_D\sum\limits_{k,j=1}^2\big(\partial_k((\F\cdot\nabla) u_\eps^j)\big)^2 \,dxd\tau \\
= & \ \frac{1}{2}\int\limits_{t_1}^{t_2}\int\limits_D \partial^2_{ks} u_\eps^j\partial^2_{kl} u_\eps^j F^sF^l\,dxd\tau 
+\int\limits_{t_1}^{t_2}\int\limits_D \partial^2_{ks} u_\eps^j\partial_l u_\eps^j F^s\partial_k F^l \,dxd\tau \\
&\ +\frac{1}{2}\int\limits_{t_1}^{t_2}\int\limits_D\partial_s u_\eps^j\partial_l u_\eps^j\partial_k F^s\partial_k F^l \,dxd\tau \\
=&\ T_{11}+T_{12}+\frac{1}{2}\int\limits_{t_1}^{t_2}\int\limits_D|\nabla u_\eps\cdot\nabla \F|^2\,dxd\tau = T_{11}+T_{12}+S_{4}.
\end{align*}
  
We work with the term $T_{11}$. In the first step, we integrate by parts with respect to $x_s$. In the second step, we integrate by parts with respect to $x_k$ only in the first of the new integrals.
\begin{align*}
T_{11}& =-\frac{1}{2}\int\limits_{t_1}^{t_2}\int\limits_D \partial^3_{kls} u_\eps^j\partial_k u_\eps^j F^l F^s \,dxd\tau-
\frac{1}{2}\int\limits_{t_1}^{t_2}\int\limits_D\partial^2_{kl} u_\eps^j\partial_k u_\eps^j\partial_s(F^l F^s)\,dxd\tau\\
& =\frac{1}{2}\int\limits_{t_1}^{t_2}\int\limits_D\partial^2_{ls} u_\eps^j\partial^2_{kk} u_\eps^j F^l F^s \,dxd\tau+
\frac{1}{2}\int\limits_{t_1}^{t_2}\int\limits_D\partial^2_{l s} u_\eps^j\partial_k u_\eps^j\partial_k(F^l F^s) \,dxd\tau\\
\ \  & -\frac{1}{2}\int\limits_{t_1}^{t_2}\int\limits_D\partial^2_{kl} u_\eps^j\partial_k u_\eps^j\partial_s(F^lF^s)\,dxd\tau
=T_{13}+T_{14}+T_{15}.
\end{align*}
We see immediately that $T_{13}= -T_{8}$. 

Since $T_{14}$ is symmetric in $s$ and $l$, we write it as
\[T_{14}=\int\limits_{t_1}^{t_2}\int\limits_D\partial^2_{l s} u_\eps^j\partial_k u_\eps^j\partial_k F^l F^s \,dxd\tau.
\]
Next, we integrate by parts with respect to $x_s$ and obtain
\begin{align*}
T_{14}& =-\int\limits_{t_1}^{t_2}\int\limits_D\partial^2_{ks} u_\eps^j\partial_l u_\eps^j\partial_k F^l F^s \,dxd\tau-
\int\limits_{t_1}^{t_2}\int\limits_D\partial_l u_\eps^j\partial_k u_\eps^j\partial_s\big(\partial_k F^l F^s\big)\,dxd\tau\\
&=-T_{12}-\int\limits_{t_1}^{t_2}\int\limits_D \partial_l u_\eps^j\partial_k u_\eps^j\partial_k F^l\partial_s F^s \,dxd\tau
-\int\limits_{t_1}^{t_2}\int\limits_D \partial_l u_\eps^j\partial_k u_\eps^j\partial^2_{ks} F^l F^s \,dxd\tau\\
&=-T_{12}-\int\limits_{t_1}^{t_2}\int\limits_D \nabla u_\eps:(\nabla u_\eps\cdot\nabla \F )\cdot\diverg \F \,dxd\tau-2\,T_{10}=-T_{12}+S_{6}+2S_{7}+S_{5}.
\end{align*}
 
Observing that $\partial^2_{kl} u_\eps^j\partial_k u_\eps^j= (1/2)\partial_l(\partial_k u_\eps^j)^2$,
we conclude that
\begin{align*}
T_{15}&=-\frac{1}{4}\int\limits_{t_1}^{t_2}\int\limits_D\partial_l(|\nabla u_\eps|^2)\partial_s(F^l F^s)\,dxd\tau\\
&=
\frac{1}{2}\int\limits_{t_1}^{t_2}\int\limits_D\frac{1}{2} |\nabla u_\eps|^2 \left(\diverg (\F\cdot\diverg \F)+\diverg(\nabla \F\cdot \F)\right)\,dxd\tau\\
&=\frac{1}{2}\int\limits_{t_1}^{t_2}\int\limits_D\frac{1}{2} |\nabla u_\eps|^2 \diverg (\F\cdot\diverg \F)\,dxd\tau-R_{1}.
\end{align*}
  We obtain the following chain of equalities
\begin{align*}
T_{1}+T_{3}&=T_{8}+T_{9}+T_{11}+T_{12}+S_{4}\\
&=T_{8}+R_{1}+\frac{1}{2}S_{5}+T_{10}+T_{12}+T_{13}+T_{14}+T_{15}+S_{4}\\
&=T_{8}+R_{1}+\frac{1}{2}S_{5}+T_{10}+T_{12}-T_{8}-T_{12}+S_{6}-2T_{10}-R_{1}\\
&\qquad+\frac{1}{2}\int\limits_{t_1}^{t_2}\int\limits_D\frac{1}{2} |\nabla u_\eps|^2 \diverg (\F\cdot\diverg \F)\,dxd\tau+S_{4}\\
&=S_{4}+S_{5}+S_{6}+S_{7}+\frac{1}{2} \int\limits_{t_1}^{t_2}\int\limits_D|\nabla u_\eps|^2 \diverg (\F\cdot\diverg \F)\,dxd\tau.
\end{align*}
Therefore, we obtain that $T_{1}+...+T_{5}=S_{3}+...+S_{7}$,
as desired.
\qed
\end{proof}

\begin{remark}
In the proof, we explicitly use the regularity of $u_\eps$. In the term $T_{11}$, there appear the third-order partial derivatives of $u_{\eps}$. Since they exist in the classical sense, they do not require any particular treatment.
\end{remark}
\begin{remark}
For the noise of the form 
$\sum_{k=1}^N \F_k(x)\circ dB_t^k,
$\!
the computations are essentially the same. The function~$\psi$ and the matrix $\Psi$ in~\eqref{Extra} are equal to the sums of the $\psi_{k}$'s and $\Psi_{k}$'s corresponding to each $\F_k$. This is due to the independence of the driving Brownian motions. If the noise is an infinite sum, the expressions are more complicated, cf. Proposition~6.1 in~\cite{RW}.
\end{remark}
 
\begin{remark} Consider the special case when the external field is given by $\F=(F(x_1),0)$. We compute explicitly that
\begin{multline*}
\int\limits_D e_\eps(u_\eps)\psi(x)+\trace(\nabla u_\eps\cdot\Psi(x)\cdot\nabla u_\eps^\top)\,dx\\
=\frac{1}{2}\int\limits_D (f_\eps(u_\eps),\partial_1u_\eps)\cdot (F\partial_1F)\,dx+\frac{1}{2}
\int\limits_D |\partial_1u_\eps|^2\cdot |\partial_1 F|^2\,dx.
\end{multline*}
This example shows that the noise has a non-trivial impact on the energy evolution. The first term on the right-hand side can be absorbed into the other terms in the corresponding equation on $E_\eps$. Still, the second term is positive and scales in $\eps$ like the Ginzburg-Landau energy. If $\F$ is constant, this problem does not occur: the whole extra term~\eqref{Extra} is zero. The choice of a constant forcing is, however, incompatible with the technique we used in the proof of Theorem~\ref{existence}. 
\end{remark}

\subsection{Energy estimates and corollaries}
For the parabolic Ginzburg-Landau equation, i.e., for \eqref{SGL_Strat} with $\F=0$, the energy of the solution decreases with time. Therefore, the energy at the time $t=0$ provides a sufficient control on the energy at later times. If the forcing is nontrivial, this is not true anymore. We now show that a certain control over the energy can be re-established via the Gronwall argument.

\begin{proposition}\label{EnergyEstimate}
There exists a constant $K_{1}$ depending on $||\F||_{C^2}$ but not on the time-horizon $T$ and $\eps\in(0,1)$ such that
\[ \MExp\big[E_\eps(u_\eps(t))\big]\leqslant \exp(K_{1}t)\cdot E_{\eps}(u_{\eps}^{0})
\]
holds for all $t\in[0,T]$ and $\eps\in(0,1)$. 
\end{proposition} 

\proof
We consider~\eqref{Energy_Evolution} with $t_{1}=0, \, t_{2}=t$ and take the expectation on both sides. This gives the identity 
\begin{multline*}
\MExp\Big[E_\eps(u_\eps(t))+ \log(1/\eps)\int\limits_0^t\int\limits_D |f_\eps(u_{\eps})|^2\,dxd\tau\Big]\\
=E_\eps(u_\eps^{0})
+\MExp\Big[\,\int\limits_0^t\int\limits_D e_\eps(u_\eps)\psi(x)+\trace(\nabla u_\eps\cdot\Psi(x)\nabla u_\eps^\top)\,dxd\tau\Big].
\end{multline*}
We have that $\MExp[E_\eps(u_\eps^{0})]=E_\eps(u_\eps^{0})$ because the initial data is non-random. 
The It\^{o} integral does not contribute because it is a martingale. Since the functions $\psi(x)$ and $\Psi(x)$ depend only on the field $\F$ and its derivatives up to the second order, the integrand on the right-hand side can be estimated from above as follows:
\[\Big|\int\limits_0^t\int\limits_D e_\eps(u_\eps)\psi(x)+\trace(\nabla u_\eps\cdot\Psi(x)\cdot \nabla u_\eps^\top)\,dxd\tau\Big|
\leqslant K_{1}(||\F||_{C^2})\int\limits_0^t E_\eps(u_\eps(\tau))\,d\tau.
\]
With the Fubini theorem, we conclude that
\[\MExp\big[E_\eps(u_\eps(t))\big] \leqslant E_\eps(u_\eps^{0})+K_{1}\int\limits_0^t \MExp\big[E_\eps(u_\eps(\tau))\big]\,d\tau.
\]
for all $t\in[0, T]$. We apply the Gronwall lemma to the function $\MExp\big[E_\eps(u_\eps(t))\big]$ and so obtain the result.
\qed

We reformulate and extend the statement of Proposition~\ref{EnergyEstimate} for the rescaled energy functional $\mathcal E_{\eps}(t)$ \eqref{eq:RescaledEnergyDef}. That will be more convenient for the next section.
  \begin{corollary}\label{EnergyBound}
Let $K_{1}$ be the constant found in Proposition~\ref{EnergyEstimate}. Then for the process~$\mathcal{E}_\eps(t)$, there holds
\begin{equation}\label{eq:RescaledEnergyEst}
\sup\nolimits_{\eps\in(0,1)}\sup\nolimits_{t\in[0,T]}\MExp\big[\mathcal{E}_\eps(t)\big]\leqslant \exp(K_{1}T)\cdot \sup\nolimits_{\eps\in(0,1)}\mathcal{E}_\eps(0)<\infty.
\end{equation}

Moreover, the process $(\mathcal{E}_{\eps}(t))^2$ satisfies
\begin{equation}\label{eq:RescaledEnergySquared}
\sup\nolimits_{\eps\in(0,1)}\sup\nolimits_{t\in[0,T]}\MExp\big[(\mathcal{E}_\eps(t))^{2}\big]<\infty.
\end{equation}
\end{corollary}
\proof
The first claim follows immediately from Proposition~\ref{EnergyEstimate}. The quantity~$\mathcal{E}_\eps(0)$ is uniformly bounded in $\eps$, by virtue of~\eqref{eq:EnergyAtTimeZero}.

To prove \eqref{eq:RescaledEnergySquared}, we argue essentially as in the proof of Proposition~\ref{EnergyEstimate}. We start with the It\^o equation for~$(\mathcal E_{\eps}(t))^{2}$. The classical It\^o formula (\cite{Karatzas_Shreve}, Theorem~3.3) implies that
\[d (\mathcal E_{\eps}(t))^{2}=2\mathcal E_{\eps}(t)\,d\mathcal E_{\eps}(t)+d\,\langle \mathcal E_{\eps}(t)\rangle_{t}.
\]
The equation for $\mathcal E_{\eps}(t)$ is simply the equation for $ E_{\eps}(u_{\eps}(t))$ multiplied by $k_{\eps}$. Therefore, the quadratic variation is given by
\[d\,\langle \mathcal E_{\eps}(t)\rangle_{t}=k_{\eps}^{2}\big(\int\limits_{D}(f_{\eps}(u_{\eps}), (\F\cdot\nabla )u_{\eps})\,dx\big)^{2}dt.
\]
The equation on $(\mathcal E_{\eps}(t))^{2}$ thus reads  
\begin{align}\label{eq:ItoForSquare}
(\mathcal E_{\eps}(t_{2}))^{2}=& \ (\mathcal E_{\eps}(t_{1}))^{2}-2\int\limits_{t_{1}}^{t_{2}}\mathcal E_{\eps}(\tau)\cdot \int\limits_D |f_\eps(u_\eps)|^2\,dxd\tau\nonumber\\
&-2k_{\eps}\int\limits_{t_{1}}^{t_{2}} \mathcal E_{\eps}(\tau)\cdot \int\limits_D (f_\eps(u_\eps), (\F\cdot\nabla) u_\eps)\,dx dB_\tau\nonumber\\
&+2k_{\eps}\int\limits_{t_{1}}^{t_{2}}\mathcal E_{\eps}(\tau)\cdot\int\limits_D e_\eps(u_\eps)\psi(x)+\trace(\nabla u_\eps\cdot\Psi(x)\cdot\nabla u_\eps^\top)\,dxd\tau\nonumber\\
&+k_{\eps}^{2}\int\limits_{t_{1}}^{t_{2}}\big(\int\limits_{D}(f_{\eps}(u_{\eps}), (\F\cdot\nabla )u_{\eps})\,dx\big)^{2}\,d\tau.
\end{align}
We easily see that 
\[\Big|k_{\eps}\int\limits_{t_{1}}^{t_{2}}\mathcal E_{\eps}(\tau)\cdot\int\limits_D e_\eps(u_\eps)\psi(x)+\trace(\nabla u_\eps\cdot\Psi(x)\cdot\nabla u_\eps^\top)\,dxd\tau \Big|\leqslant C \int\limits_{t_{1}}^{t_{2}} (\mathcal E_\eps(\tau))^{2}\, d\tau.
\]
For a matrix-valued function $A(x)=(a^{ij}(x))$, the divergence $\diverg A(x)$ is a vector field with the components given by~$(\diverg A)^{j}(x)=\partial_{i}a^{ij}(x)$.
Using this notation and integrating by parts twice, we obtain that
\begin{multline*}
\int\limits_{D}(f_{\eps}(u_{\eps}), (\F\cdot\nabla )u_{\eps})\,dx=\int\limits_{D}(\F, \diverg (\nabla u_{\eps}\otimes \nabla u_{\eps})-\nabla e_{\eps}(u_{\eps}))\,dx\\
=\int\limits_{D}e_{\eps}(u_{\eps})\cdot \diverg \F - \nabla \F:(\nabla u_{\eps}\otimes \nabla u_{\eps})\,dx.
\end{multline*}
Consequently, there holds
\[
k_{\eps}^{2}\int\limits_{t_{1}}^{t_{2}}\big(\int\limits_{D}(f_{\eps}(u_{\eps}), (\F\cdot\nabla )u_{\eps})\,dx\big)^{2}\,d\tau\leqslant 
C \int\limits_{t_{1}}^{t_{2}} (\mathcal E_\eps(\tau))^{2}\, d\tau.
\]
 
We set $t_{1}=0$, $t_{2}=t$ in \eqref{eq:ItoForSquare} and obtain that
\[
\MExp\big[(\mathcal E_\eps(t))^{2}\big]\leqslant (\mathcal E_\eps(0))^{2}+C \int\limits_0^t \MExp\big[(\mathcal E_\eps(\tau))^{2}\big]\, d\tau
\]
for any $t\in [0, T]$.
It remains to apply the Gronwall lemma, as in the proof of Proposition~\ref{EnergyEstimate}.
\qed

We finally obtain some information on the solution.
\begin{corollary}
There exists a positive constant $K_{2}=K_{2}(T)$ such that 
\begin{equation}\label{eq:L2Bound}
\sup\nolimits_{\eps\in(0,1)}\sup\nolimits_{t\in[0,T]} \MExp\big[ \lVert u_\eps(t)\rVert_{L^{2}}^{2}\big]\leqslant K_{2}.
\end{equation}
\end{corollary}
\proof
For every $t\in[0,T]$ and every $\omega\in\Omega $ we have that
\begin{multline*}
\int\limits_{D}|u_{\eps}(t, \omega, x)|^{2}\,dx=\int\limits_{D}(|u_{\eps}|^{2}-1)\,dx+C\leqslant C\big(\int\limits_{D}(|u_{\eps}|^{2}-1)^{2}\,dx\big)^{1/2}+C\\
\leqslant  C(\eps^2\log(1/\eps)\mathcal E_{\eps}(t))^{1/2}+C\leqslant \eps^{2}\log(1/\eps)\mathcal E_{\eps}(t)+C.
\end{multline*}
 We use the embedding $L^2\hookrightarrow L^1$\! in the second step and Young's inequality in the last step. All constants that we denote with $C$ depend only on the domain~$D$. Now \eqref{eq:RescaledEnergyEst} gives the estimate
\[\sup\nolimits_{\eps\in(0,1)}\sup\nolimits_{t\in[0,T]} \MExp\big[ \lVert u_\eps(t)\rVert_{L^{2}}^{2}\big]\leqslant \sup\nolimits_{\eps\in(0,1)} \{\eps^{2}\log(1/\eps)\mathcal E_{\eps}(0)\}\cdot \exp\{K_{1}T\}+C.
\]
Since we consider $\eps\in(0,1)$, the function $\eps^{2}\log(1/\eps)$ is uniformly bounded, and we obtain~\eqref{eq:L2Bound} by recalling~\eqref{eq:EnergyAtTimeZero}.
\qed

\section{Tightness for the Jacobian and the stochastic vortices}\label{sec:tightness}

In this section, we prove Theorems \ref{tightness} and~\ref{structure}. Both proofs rely on the deterministic results of Jerrard and Soner from~\cite{JerrardSoner}. We start with the result on tightness.
\begin{proof}[Proof of Theorem~\ref{tightness}.]
We fix some $t\in[0,T]$.
The tightness of the family $(\mu_\eps(t))_{\eps}$ follows from \eqref{eq:RescaledEnergyEst} and the compact embedding 
$L^1\Subset (C_0^{0,\alpha})^*\!.
$
Indeed, by the Chebyshev inequality we have that 
\[\pP\{\lVert \mu_{\eps}(t)\rVert_{L^{1}}\geqslant \lambda\}\leqslant (1/\lambda)\MExp\big[\lVert \mu_{\eps}(t)\rVert_{L^{1}}\big]=(1/\lambda)\MExp\big[\mathcal E_{\eps}(t)\big]
\]
for any $\lambda>0$.
The estimate \eqref{eq:RescaledEnergyEst} implies that there exists a constant $C$ such that 
\[\sup\nolimits_{\eps\in(0,1)}\pP\{\lVert \mu_\eps(t)\rVert_{L^1}\geqslant \lambda\}\leqslant C/\lambda\]
for any $\lambda>0$.
By choosing $\lambda$ large enough, we can make the right-hand side of the last inequality arbitrarily small. Hence, the probability that all $(\mu_{\eps}(t))_{\eps}$ belong to a bounded subset of $L^{1}$ can be made arbitrarily close to one. This subset is compact in the topology of~$(C_0^{0,\alpha})^*$\!. The family $(\mu_{\eps}(t))_{\eps}$ is thus by definition tight on $(C_0^{0,\alpha})^*$.

The tightness of the Jacobians will follow with the same argument applied to the estimate
\begin{equation}\label{Jacobian_expectation}
\sup\nolimits_{\eps\in(0,1)}\sup\nolimits_{t\in[0,T]}\MExp\big[\lVert J(u_\eps(t))\rVert_{(C_0^{0,\alpha})^*}\big]\leqslant C
\end{equation}
and the compact embedding  
 $(C_0^{0,\alpha_1})^*\Subset (C_0^{0,\alpha_2})^*$ for $0\leqslant\alpha_1<\alpha_2\leqslant 1$. 

We obtain~\eqref{Jacobian_expectation} by applying, pointwise in~$\omega$, the results on the Jacobian derived in~\cite{JerrardSoner}. These results are by themselves quite delicate, but we use them below in a very straightforward manner. By Theorem~\ref{existence}, the solution~$u_{\eps}(x, t,\omega)$ of \eqref{SGL_Strat} almost surely belongs to the space $C_t^{\beta}C_x^{3,\beta}(\overbar {D_{T}})$ with a $\beta\in(0, 1/2)$. Let $A\subset \Omega$ be the set of full measure on which this holds. We consider arbitrary $\omega\in A$ and $t\in[0, T]$. For them, $u_\eps(x, t,\omega)$ belongs to the Sobolev space~$H^{1}$. We can thus apply \cite{JerrardSoner}, Proposition~3.2 to the Jacobian of $u_{\eps}(x, t,\omega)$: $J(u_\eps(x, t,\omega))$ can be decomposed into a sum 
\[J(u_\eps(x, t,\omega))=J_0^\eps(t,\omega)+J_1^\eps(t,\omega).
\]
For the first part $J_0^\eps$, we have the estimate
\[\lVert J_0^\eps(t,\omega)\rVert_{(C^0)^*}\leqslant C\,\mathcal E_\eps(u_{\eps}(t,\omega)).
\]
For the remaining part $J_1^\eps$, we have the estimate
\[\lVert J_1^\eps(t,\omega)\rVert_{(C_0^{0,\alpha})^*}\leqslant q_{\alpha}(\eps)\cdot\mathcal E_\eps(u_{\eps}(t,\omega))
\]
that holds for every $\alpha\in(0, 1]$, with $q_{\alpha}(\eps)\to 0$ for $\eps\to 0$. Therefore, for every $\alpha \in(0, 1]$, for every~$t\in[0,T]$, and for every $\omega\in A$, we obtain that
\[\lVert J(u_\eps(t,\omega))\rVert_{(C_0^{0,\alpha})^*}\leqslant (C+q_{\alpha}(\eps))\cdot \mathcal E_\eps(u_{\eps}(t,\omega))\leqslant (C+1)\mathcal E_\eps(u_{\eps}(t,\omega)).
\]
The constant $C$ does not depend on $t$, $\omega$, or $\eps$. Since $\omega$ and $t$ were arbitrary, we may take the expectation and the supremum on both sides in the inequality above. Now \eqref{Jacobian_expectation} follows from \eqref{eq:RescaledEnergyEst}.
\qed
\end{proof}

\begin{remark}
In~\cite{JerrardSoner}, the results on the compactness of the Jacobian are formulated for the space $(C_{c}^{0,\alpha})^{*}$, but the proofs presented are valid in~$(C_{0}^{0,\alpha})^{*}$ as well.
\end{remark}

Now we have everything in place to prove Theorem~\ref{structure}.
\begin{proof}
\proof [of Theorem~\ref{structure}.]
We fix a $t\in[0, T]$ and an $\alpha\in(0, 1)$. We set $q:=2/(1+\alpha)$ and~$p:=2/(1-\alpha)$. Then we have that 
$1/p+1/q=1$. 

We first note that the three embeddings 
\begin{equation*}\label{eq:embeddingLarge}
W_{0}^{1,p}\hookrightarrow C_{0}^{0, \alpha}\hookrightarrow C_{0}^{0, \gamma}\hookrightarrow L^{2}
\end{equation*}
are continuous for $\gamma\in [0, \alpha]$. The first is the special case of the Sobolev embedding (\cite{Brezis_FAPDE}, Theorem~9.17). The second follows from the Ascoli-Arzela theorem. This embedding is \textit{compact} for $\gamma <\alpha$, cf.~\cite{Adams}, Theorem~1.31. The third embedding is elementary. By duality we conclude that the embedding 
\begin{equation}\label{eq:embeddingComp}
L^{2}\Subset W^{-1,q}
\end{equation}
is compact, and the embedding 
\begin{equation}\label{eq:embeddingCont}
(C_{0}^{0, \alpha})^{*}\hookrightarrow W^{-1,q}
\end{equation}
is continuous.

The family $(u_{\eps}(t))_{\eps}$ is tight on the space $W^{-1,q}$, due to \eqref{eq:L2Bound} and \eqref{eq:embeddingComp}. This follows from the argument with the Chebyshev inequality, as in the proof of Theorem~\ref{tightness}~(i). 

On $(C_{0}^{0, \alpha})^{*}$, we consider weakly convergent sequences $(J(u_{\eps_{n}}(t)))_{\eps_{n}}$ and $(\mu_{\eps_{n}}(t))_{\eps_{n}}$ that correspond to the same sequence $\eps_{n}\to 0$. The sequences converge weakly in the space~$W^{-1,q}$\! as well, because of~\eqref{eq:embeddingCont} and the mapping theorem (\cite{Billingsley_Conv}, Theorem~2.7). We also assume without loss of generality that the sequence $(u_{\eps_{n}}(t))_{\eps_{n}}$ converges weakly on~$W^{-1,q}$\!. The limiting probability measures on $W^{-1,q}$\! are denoted by $\mathcal P_{J}$, $\mathcal P_{\mu}$, and $\mathcal P_{u}$. By Theorem~2.8 of \cite{Billingsley_Conv}, the sequence $((J(u_{\eps_{n}}(t)), \mu_{\eps_{n}}(t), u_{\eps_{n}}(t)))_{\eps_{n}}$ converges weakly to the product measure $\mathcal P_{J}\times\mathcal P_{\mu}\times\mathcal P_{u}$ on the space 
\[\mathbb W:=W^{-1,q}\times W^{-1,q}\times W^{-1,q}\!.\]

The space $\mathbb W$ is separable, because $W^{-1,q}$\! is. We can thus apply the Skorokhod representation theorem (\cite{Ethier_Kurtz}, Theorem~3.1.8) to $((J(u_{\eps_{n}}(t)), \mu_{\eps_{n}}(t), u_{\eps_{n}}(t)))_{\eps_{n}}$. We obtain a new probability space $(\tilde \Omega, \mathfrak {\tilde F}, \tilde \pP)$ and on it, a new sequence~$((\tilde J_{n}, \tilde \mu_{n}, \tilde u_{n}))$ of $\mathbb W$-valued variables with the following properties. The sequence converges $\tilde \pP$-almost surely to a random variable $(\tilde J,\tilde \mu, \tilde u)$; for every $n\in \N$, we have that $\mathcal L(\tilde J_{n})=\mathcal L(J(u_{\eps_{n}}(t)))$, $\mathcal L(\tilde \mu_{n})=\mathcal L(\mu_{\eps_{n}}(t))$, and $\mathcal L(\tilde u_{n})=\mathcal L(u_{\eps_{n}}(t))$; finally, 
$\mathcal L(\tilde J)=\mathcal P_{J}$, $\mathcal L(\tilde \mu)=\mathcal P_{\mu}$, and $\mathcal L(\tilde u)=\mathcal P_{u}$ .

The Borel subsets of $H^{1}$ are Borel subsets of $W^{-1,q}$\!, by virtue of the Kuratowski theorem (\cite{VakhaniaTarieladzeChobanyan}, Theorem~1.1). Since the laws of $\tilde u_{n}$ and $u_{\eps_{n}}(t)$ are the same, we actually have that 
\[\tilde \pP\,\{\tilde u_{n}\in H^{1}\}=\pP\,\{ u_{\eps_{n}}(t)\in H^{1}\}=1.
\]
The mapping $u\mapsto \det\nabla u$ from $H^{1}$ to $L^{1}$, and thus to $W^{-1,q}$, is continuous. Using the Kuratowski theorem and the equality of laws once again, we obtain that 
\[\mathcal L(\tilde J_{n}-\det\nabla \tilde u_{n})=\mathcal L (J( u_{\eps_{n}}(t))-\det\nabla( u_{\eps_{n}}(t))).\]
The distribution on the right-hand side is just the Dirac measure centered at zero in the space $W^{-1,q}$. Therefore, we conclude that
\[\tilde \pP\,\{\tilde J_{n}=\det\nabla \tilde u_{n}\}=\pP\,\{J( u_{\eps_{n}}(t))=\det\nabla( u_{\eps_{n}}(t))\}=1,
\]
for every $n\in \N$. In other words, $(\tilde J_{n})$ \textit{is} the sequence of the Jacobians associated to~$(\tilde u_{n})$, with probability one. Therefore, for every $n \in \N$, $\tilde{J}_{n}$ belongs to the space~$L^{1}$\!, almost surely. 

Moreover, we have that  
\begin{equation}\label{eq:UTildeExp}
\tilde \MExp\big[(\mathcal{E}_{\eps_{n}}(\tilde u_{n}))^{2}\big]=\MExp\big[(\mathcal{E}_{\eps_{n}}( u_{\eps_{n}}(t)))^{2}\big]\leqslant C.
\end{equation}
Finally, $(\tilde \mu_{n})$ is the sequence of the rescaled energy densities associated to~$(\tilde u_{n})$, with probability one:
\[\tilde \pP\,\{\tilde \mu_{n}=k_{\eps_{n}}e_{\eps_{n}}(\tilde u_{n})\}=1.
\]
 The justification of the last two statements is the same as in the case of $(\tilde J_{n})$.

We obtain from \eqref{eq:UTildeExp} with the help of Fatou's lemma that 
\begin{equation}\label{eq:Fatou}
\tilde \MExp\big[\liminf _{n\to \infty}(\mathcal{E}_{\eps_{n}}(\tilde u_{n}))^{2}\big]\leqslant C.
\end{equation}
The Chebyshev inequality now implies that 
$\tilde \pP\,\{\liminf _{n\to \infty}\mathcal{E}_{\eps_{n}}(\tilde u_{n})\geqslant m\}\leqslant C/(m^{2})
$
for every $m\in \N$.
This means that the series $\sum_{m=1}^{\infty}\tilde \pP\,\{\liminf _{n\to \infty}\mathcal{E}_{\eps_{n}}(\tilde u_{n})\geqslant m\}$ converges. Therefore, by the Borel-Cantelli lemma, for almost every $\tilde \omega\in \tilde \Omega$ there exists an $m=m(\tilde \omega)$ such that 
\[\liminf _{n\to \infty}\mathcal{E}_{\eps_{n}}(\tilde u_{n}(\tilde \omega))\leqslant m.
\]
We consider an $\tilde \omega$ for which the inequality above holds. There exists a subsequence of~$(\mathcal{E}_{\eps_{n}}(\tilde u_{n}(\tilde \omega)))$ that converges to the limes inferior. The corresponding sequence of indices~$n_{l}\to \infty$ does depend on $\tilde \omega$.
We may assume without loss of generality that the sequence $(\tilde J_{n_{l}}(\tilde \omega))=(J(\tilde u_{n_{l}}(\tilde\omega)))$ converges to $\tilde J(\tilde \omega)$. Similarly, we may assume that~$(\tilde \mu_{n_{l}}(\tilde \omega))=(\mu_{\eps_{n_{l}}}(\tilde u_{n_{l}}(\tilde\omega)))$ converges to $\tilde \mu(\tilde\omega)$. Indeed, both sequences converge almost surely. Now we may apply Theorem 3.1 of~\cite{JerrardSoner} to $(\tilde J_{n_{l}}(\tilde \omega))$ and $(\tilde \mu_{n_{l}}(\tilde \omega))$, and obtain all claims of Theorem~\ref{structure}. First, both sequences~$(\tilde J_{n_{l}}(\tilde \omega))$ and $(\tilde \mu_{n_{l}}(\tilde \omega))$ subconverge (without loss of generality, along the same sequence of indices) not only in the topology of $W^{-1,q}$, but also in that of~$(C_{0}^{0,\alpha})^{*}$. The structure of the limit $\tilde J(\tilde \omega)$ is given by the deterministic result as well. We have that
\begin{equation}\label{eq:representingJacobian}
\tilde J(\tilde \omega)=\pi\sum\limits_{k=1}^{N(\tilde \omega)}d_{k}(\tilde \omega)\delta_{a_{k}(\tilde \omega)},
\end{equation}
with $N(\tilde \omega)\in \N$, $d_{k}(\tilde \omega)\in \Z$ and $a_{k}(\tilde \omega)\in D$. We also have that  
$\tilde J(\tilde \omega)\ll \tilde \mu(\tilde \omega)$ with $\big|\frac{d\tilde J(\tilde \omega)}{d\tilde \mu(\tilde \omega)}(x)\big|\leqslant 1$ for $\tilde \mu(\tilde \omega)$-almost all~$x\in D$. Finally, we have the estimate  
 \begin{equation}\label{eq:representingEstimate}
\lVert\tilde J(\tilde \omega)\rVert_{(C_{0}^{0})^{*}}=\pi\sum_{k=1}^{N(\tilde \omega)}|d_{k}(\tilde \omega)|\leqslant \lim _{l\to \infty}\mathcal{E}_{\eps_{n_{l}}}(\tilde u_{n_{l}}(\tilde \omega))\leqslant \liminf _{n}\mathcal{E}_{\eps_{n}}(\tilde u_{n}(\tilde \omega)).
\end{equation}
 Since $\tilde \omega$ has been chosen arbitrarily, the identity \eqref{eq:representingJacobian} and the estimate \eqref{eq:representingEstimate} do hold for~$\tilde \pP$-almost all $\tilde \omega$. When we take expectation on both sides of~\eqref{eq:representingEstimate}, we get \eqref{eq:ExpectationOfNumberOfVortices}. This concludes the proof.
\qed
\end{proof}

\begin{remark}
Theorem~3.1 in~\cite{JerrardSoner} claims that the estimate $\frac{d\tilde J(\tilde \omega)}{d\tilde \mu(\tilde \omega)}(x)\leqslant 1$ is true for $\tilde \mu(\tilde \omega)$-almost all $x\in D$. However, the proof in~\cite{JerrardSoner} actually leads to a stronger estimate $\big|\frac{d\tilde J(\tilde \omega)}{d\tilde \mu(\tilde \omega)}(x)\big|\leqslant 1$, which we use above. 
\end{remark}

\textbf{Acknowledgements}
We would like to thank anonymous referees for their comments which helped to improve the manuscript. 

This work contains some results from the PhD thesis~\cite{RWTHDiss} of the first author. The first author has been supported by the German Academic Exchange Service (DAAD) grant A/10/86352.


\end{document}